\documentclass[12pt,a4paper]{amsart}

\usepackage{amsmath,amsthm,amssymb,latexsym,a4wide,tikz,multicol,tikz-cd}
\usepackage{arydshln,multirow,soul}
\usepackage{xcolor,enumerate,hyperref,cleveref}
\usepackage[active]{srcltx}
\usepackage{mathtools,stmaryrd}
\tikzset{font=\small}
\usetikzlibrary{matrix,arrows}
\usetikzlibrary{positioning}
\usetikzlibrary{patterns}

\newtheorem{theorem}{Theorem} [section]
\newtheorem{lemma}[theorem]{Lemma}
\newtheorem{corollary}[theorem]{Corollary}
\newtheorem{proposition}[theorem]{Proposition}

\newtheorem{example}[theorem]{Example}
\newtheorem{remark}[theorem]{Remark}
\newtheorem{definition}[theorem]{Definition}

\subjclass[2010]{20M10,  20M18, 20M05, 20M99, 20F05}

\keywords{expansion, prefix expansion, Szendrei expansion, restriction monoid, inverse monoid, partial action, premorphism}

\usepackage[active]{srcltx}

\usepackage{enumerate}
\numberwithin{equation}{section}

\title[Expansions of monoids]{Two-sided expansions of monoids}

\author{Ganna Kudryavtseva}
\address{G. Kudryavtseva: University of Ljubljana,
Faculty of Civil and Geodetic Engineering, Jamova cesta~2, SI-1000 Ljubljana, Slovenia \and Institute of Mathematics, Physics and Mechanics, Jadranska ulica 19, SI-1000 Ljubljana, Slovenia}
\email{ganna.kudryavtseva\symbol{64}fgg.uni-lj.si, ganna.kudryavtseva\symbol{64}imfm.si}

\thanks{The author was partially supported by  ARRS grant P1-0288.}

\begin{document}

\begin{abstract}  
We initiate the study of expansions of monoids in the class of {\em two-sided restriction monoids} and show that generalizations of the Birget-Rhodes prefix group expansion, despite the absence of involution, have rich structure close to that of relatively free {\em inverse} monoids. 

For a monoid $M$ and a class of partial actions of $M$, determined by a set, $R$, of identities, we define
 ${\mathcal{FR}}_{R}(M)$ to be the universal $M$-generated two-sided restriction monoid with respect to partial actions of $M$ determined by $R$. This is an $F$-restriction monoid which (for a certain $R$) generalizes the Birget-Rhodes  prefix expansion $\widetilde{G}^{{\mathcal R}}$ of a group $G$.  Our main result provides a coordinatization of ${\mathcal{FR}}_{R}(M)$ via a partial action  product of the idempotent semilattice $E({\mathcal{FI}}_{R}(M))$ of a similarly defined {\em inverse} monoid, partially acted upon by $M$.  
The result by Fountain, Gomes and Gould on the structure of the free two-sided restriction monoid is recovered as a special case of our result. 

We show that some properties of ${\mathcal{FR}}_{R}(M)$  agree well with suitable properties of  $M$, such as being cancellative or embeddable into a group. 
We observe that if $M$ is an inverse monoid, then ${\mathcal{FI}}_{R_s}(M)$, the free inverse monoid with respect to strong premorphisms, is isomorphic to the Lawson-Margolis-Steinberg generalized prefix expansion $M^{pr}$.  This gives a presentation of $M^{pr}$ and  leads to a model for  ${\mathcal{FR}}_{R_s}(M)$ in terms of the known model for $M^{pr}$. 

\end{abstract}

\maketitle

\section{Introduction}\label{s0:introduction}

The  {\em Birget-Rhodes prefix expansion of a group} \cite{BR89} and, in particular, its model proposed by Szendrei \cite{Szendrei}, is a construction that assigns to an arbitrary group $G$ an $E$-unitary inverse monoid, $\widetilde{G}^{{\mathcal R}}$, which has a number of interesting properties and important applications.

There exists an extensive literature on the extension of this construction from groups and inverse semigroups \cite{LMS} to monoids and their generalizations \cite{FG, FGG_enlargements, Gomes1, GH, H, H1}. As an analogue of $\widetilde{G}^{{\mathcal R}}$ all works mentioned (as well as \cite{Gomes_Gould, G_graph} which generalize the Margolis-Meakin graph expansions \cite{MM}) output  {\em one-sided restriction monoids}. Given that  {\em two-sided restriction monoids} (see Section \ref{sect:prelim} for definition and discussion of these objects) inherit more structure present in inverse monoids than one-sided ones, in particular, left-right dual properties, it looks somewhat surprising that expansions of monoids in the class of two-sided restriction monoids have not  so far been systematically studied. The present paper initiates such a study by showing that they have rich structure close to that of relatively free  inverse monoids. 

Szendrei's model  \cite{Szendrei} represents the elements of $\widetilde{G}^{{\mathcal R}}$ by pairs $(A, g)$ with $A\subseteq G$ being a finite subset containing $1$ and $g$.  It is this model of $\widetilde{G}^{{\mathcal R}}$ that found generalizations to monoids. Applied to a monoid $M$, it produces a one-sided restriction monoid, $Sz(M)$,  the {\em Szendrei expansion} of $M$. In the group case,  we have $(A,g)^+=(A, g)(A,g)^{-1} = (A, 1)$ and $(A,g)^*= (A,g)^{-1}(A, g) = (g^{-1}A, 1)$. In the monoid case, one puts $(A,m)^+=(A,1)$, which makes $Sz(M)$ a one-sided restriction monoid.  However, since monoids do not have involution, there is no similar way to extend $(A,g)^*$  to the monoid case.

Szendrei's model of $\widetilde{G}^{{\mathcal R}}$ appears  in Exel's work \cite{Exel} as an inverse monoid being a quotient of a certain inverse monoid,  ${\mathcal{S}}(G)$, given by generators and relations, determined by the property that any partial action of $G$ can be lifted to an action of ${\mathcal{S}}(G)$. Kellendonk  and Lawson \cite{KL} completed Exel's work by showing that ${\mathcal{S}}(G)$ is in fact isomorphic to $\widetilde{G}^{{\mathcal R}}$. Due to the correspondence between partial actions and premorphisms,  $\widetilde{G}^{{\mathcal R}}$ is uniquely determined by its universal property: the map $G\to \widetilde{G}^{{\mathcal R}}$, $g\mapsto (\{1,g\},g)$, is a premorphism universal amongst all premorphisms from $G$ to inverse monoids. 

In this paper, we take the universal property of  $\widetilde{G}^{{\mathcal R}}$, extended to the setting where $M$ is a monoid,  as the definition of an analogue of $\widetilde{G}^{{\mathcal R}}$. In order to provide a unified treatment of various classes of partial actions of monoids by partial bijections  that extend partial actions of groups (or their relaxed version omitting the requirement that the inverse element be preserved), we introduce the notions of a $(\cdot\,,^*,^+,1)$-identity over $X^*$ and of a premorphism satisfying such an identity (see Section \ref{sect:identities}). The only requirement we need to impose for our arguments to go through, is that identities under consideration be satisfied by all monoid morphisms. We call such identities {\em admissible}.  This leads to the notion of a class of partial actions of  monoids  defined by admissible identities. For example, the classes of all partial actions, of strong partial actions and of actions are defined by admissible identities.

For a monoid $M$ and  a set,~$R$, of admissible identities, we define ${\mathcal{FR}}_{R}(M)$ to be the universal $M$-generated two-sided restriction monoid such that the inclusion map $\iota_{{\mathcal{FR}}_{R}(M)}\colon M\to {\mathcal{FR}}_{R}(M)$ is a premorphism satisfying identities $R$. 
Our crucial observation  is that the projection semilattice $P({\mathcal{FR}}_{R}(M))$ turns out to be isomorphic to the idempotent semilattice $E({\mathcal{FI}}_{R}(M))$ of the similarly defined universal {\em inverse} monoid ${\mathcal{FI}}_{R}(M)$. Composing the inclusion premorphism $\iota_{{\mathcal{FI}}_{R}(M)}\colon M\to {\mathcal{FI}}_{R}(M)$ with the {\em Munn representation} of ${\mathcal{FI}}_{R}(M)$, one obtains a premorphism from $M$ to $T_{E({\mathcal{FI}}_{R}(M))}$, the {\em Munn inverse monoid} of the semilattice $E({\mathcal{FI}}_{R}(M))$.
This data, via a partial action product construction, gives rise to the two-sided restriction monoid $E({\mathcal{FI}}_R(M))\rtimes M$ (for its precise definition, see Subsection \ref{subs:partial_prod}). Theorem \ref{th:main}, our main result, states that ${\mathcal{FR}}_{R}(M)$ is isomorphic to $E({\mathcal{FI}}_R(M))\rtimes M$.
In the case where $R$ defines  the class of monoid morphisms and $M=A^*$, we recover the result by Fountain, Gomes and Gould \cite{FGG} on the structure of the free $A$-generated two-sided restriction monoid. Theorem \ref{th:main} provides a precise sense of the main thesis of this paper that ${\mathcal{FR}}_{R}(M)$, despite the absence of involution, has rich structure close to that of ${\mathcal{FI}}_{R}(M)$. 

We show that properties of the monoid $M$ agree well with suitable properties of ${\mathcal{FR}}_{R}(M)$. Thus $M$ is left cancellative (resp. right cancellative or cancellative) if and only if ${\mathcal{FR}}_{R}(M)$ is right ample (resp. left ample or ample), see Corollary~\ref{cor:ample}. Furthermore, $M$ embeds into a group if and only if the canonical morphism ${\mathcal{FR}}_{R}(M)\to {\mathcal{FI}}_{R}(M)$ is injective, see Theorem~\ref{th:consequences2}. We pay special attention to the case where $M$ is an inverse monoid and show that ${\mathcal{FI}}_{R_s}(M)$, the free inverse monoid over $M$ with respect to strong premorphisms, is isomorphic to the {\em Lawson-Margolis-Steinberg  generalized prefix expansion} $M^{pr}$ of~$M$~\cite{LMS}, see~Theorem~\ref{th:presentation}. In view of Theorem~\ref{th:main} and the known model for $M^{pr}$~\cite{LMS}, this leads to a model of ${\mathcal{FR}}_{R_{s}}(M)$, the free restriction monoid over $M$ with respect to strong premorphisms. In the case where $G$ is a group, ${\mathcal{FR}}_{R_{s}}(G)$ is isomorphic to each of ${\mathcal{FI}}_{R_s}(G)$ and $\widetilde{G}^{{\mathcal R}}$, see~Corollary~\ref{cor:group}.

The paper is organized as follows. In Section \ref{sect:prelim} we recall definitions and facts on two-sided restriction semigroups needed in this paper and prove Theorem \ref{th:ample1}. In Section \ref{sect:identities} we define and discuss admissible $(\cdot\, ,^*, ^+,1)$-identities over $X^*$ and the notion of a premorphism satisfying such an identity.  In Section \ref{sect:expansions} we introduce the expansions ${\mathcal{FR}}_{R}(M)$ and study their first properties. Further, in Section~\ref{sect:F}, we state the universal $F$-restriction  property of ${\mathcal{FR}}_{R}(M)$ which generalizes a result by Szendrei \cite{Szendrei}.  In Section~\ref{sect:main} we introduce the inverse monoids ${\mathcal{FI}}_R(M)$ and the partial action products $E({\mathcal{FI}}_R(M))\rtimes M$ and then formulate Theorem \ref{th:main}. Section~\ref{sect:proof} is devoted to the  proof of Theorem \ref{th:main}.  Finally, in  Section~\ref{sect:special} we consider the special cases where $M$ is  embeddable into a group or is an inverse monoid.  

\section{Preliminaries}\label{sect:prelim} 
\subsection{Restriction monoids} In this section we recall the definition and basic properties of two-sided restriction semigroups and monoids. Our presentation  follows recent work \cite{Szendrei2013, Szendrei2014} on the subject. For more details, we refer the reader to \cite{G,H2, Cornock_phd}. 

A {\em left restriction semigroup} is an algebra $(S; \cdot \, ,  ^+)$, where $(S;\cdot \,)$ is a semigroup and $^+$ is a unary operation satisfying the following identities:
\begin{equation}\label{eq:axioms:plus}
x^+x=x, \,\,\, x^+y^+=y^+x^+, (x^+y)^+=x^+y^+, \,\,\, xy^+=(xy)^+x.
\end{equation}
Dually, a {\em right restriction semigroup} is an algebra $(S; \cdot \, , ^*)$, where $(S;\cdot \,)$ is a semigroup and $^*$ is a unary operation satisfying the following identities:
\begin{equation}\label{eq:axioms:star}
xx^*=x, \,\,\, x^*y^*=y^*x^*, \,\,\, (xy^*)^*=x^*y^*, \,\,\, x^*y=y(xy)^*.
\end{equation}

Left and right restriction semigroups are {\em one-sided restriction semigroups}.

A {\em two-sided restriction semigroup}, or just a {\em restriction semigroup}, is an algebra $(S; \cdot \, , ^*, ^+)$ of type $(2,1,1)$, where $(S;\cdot \, , ^+)$ is a left restriction semigroup, $(S;\cdot \, , ^*)$ is a right restriction semigroup and the following identities hold:
\begin{equation}\label{eq:axioms:common}
(x^+)^*=x^+,\,\,\, (x^*)^+=x^*.
\end{equation}

A restriction semigroup possessing an identity element is called a {\em restriction monoid}. A restriction  monoid is thus an algebra $(S; \cdot \, , ^*, ^+, 1)$ of type $(2,1,1,0)$. Note that the axioms imply that, in a restriction monoid, $1^*=1^+=1$. 

Morphisms and subalgebras of restriction semigroups and restriction monoids are taken with respect to the given signatures.
It is immediate from the definition that restriction semigroups and restriction monoids form varieties of algebras. By $FR(X)$ we denote the free restriction monoid over a set $X$.

Let $X$ be a set, $S$ a restriction semigroup and $\iota_S\colon X\to S$ a (not necessarily injective) map. We say that $S$ is $X$-{\em generated} via the map $\iota_S$ if $S$ is generated by $\iota_S(X)$. Let $S$ and $T$ be restriction semigroups and $\iota_S\colon X\to S$,  $\iota_T\colon X\to T$ be maps. Let, further, $\varphi\colon S\to T$ be a morphism. We say that $\varphi$ is $X$-{\em canonical}, or simply {\em canonical}, if $\varphi\iota_S = \iota_T$.
 
Let $S$ be a restriction semigroup. It follows from \eqref{eq:axioms:common} that 
$$
\{s^*\colon s\in S\}=\{s^+\colon s\in S\}.
$$
This set, denoted by $P(S)$, is closed with respect to the multiplication and is a semilattice with $e\leq f$ if and only if $e=ef=fe$ and $e\wedge f=ef$. The top element of $P(S)$ is $1$ and  $e^*=e^+=e$ holds for all $e\in P(S)$.  It is called the {\em semilattice of projections} of $S$ and its elements are called {\em projections}. A projection is necessarily an idempotent, but a restriction semigroup may contain idempotents that are not projections. 

It follows that a restriction semigroup $S$ satisfying $P(S)=S$ is necessarily a semilattice. Since any semilattice $S$ is a restriction semigroup with $e^*=e^+=e$ for all $e\in S$, restriction semigroups $S$  satisfying $S=P(S)$ are precisely semilattices. The opposite extreme are restriction semigroups $S$ satisfying $|P(S)|=1$. Then, necessarily, $S$ is a monoid and $P(S)=\{1\}$ so that  $s^*=s^+=1$ holds for all $s\in S$. Since any monoid $S$ is a restriction semigroup with $s^*= s^+= 1$ for all $s\in S$, restriction semigroups $S$ satisfying $|P(S)|=1$ are precisely monoids. A monoid $S$, looked at as a restriction semigroup with $|P(S)|=1$, is called a {\em reduced restriction semigroup}.   

The following equalities follow from the axioms and will be frequently  used in the sequel:
\begin{equation}\label{eq:mov_proj}
es = s(es)^* \, \text{ and } \, se = (se)^+s \, \text{ for all } \, s\in S  \,\, \text{ and } \, e\in P(S);
\end{equation}
\begin{equation}\label{eq:consequences}
(st)^*=(s^*t)^*,\,\,\, (st)^+=(st^+)^+ \text{ for all } \, s,t\in S;
\end{equation}
\begin{equation}\label{eq:frequent}
(se)^*=(s^*e)^* = s^*e \, \text{ and } \, (es)^+ = (es^+)^+ = es^+ \text{ for all } \, s\in S  \,\, \text{ and } \, e\in P(S).
\end{equation}

In Section \ref{sect:proof} we will need the following equalities that follow from \eqref{eq:mov_proj} and \eqref{eq:frequent}:
\begin{equation}\label{eq:fu}
((se)^+s)^*  = s^*e \, \text{ and }\,  (s(es)^*)^+ = es^+ \text{ for all } \, s\in S  \,\, \text{ and } \, e\in P(S).
\end{equation}
The {\em natural partial order} $\leq$ on a restriction semigroup  $S$ is defined, for $s,t\in S$, by $s\leq t$ if and only if there is $e\in P(S)$ such that
 $s=et$.   Elements $s,t\in S$ are said to be {\em compatible}, denoted $s\sim t$, if $st^* = ts^*$ and $t^+s = s^+t$.  The following properties related to the natural partial order and the compatibility relation will be used throughout the paper, often without reference.

\begin{lemma}\label{lem:lem1} Let $S$ be a restriction semigroup and $s,t\in S$. Then:
\begin{enumerate}
\item \label{i:b1}
  $s\leq t$ if and only if  $s=tf$ for some  $f\in P(S)$;
\item \label{i:b2} $s\leq t$ if and only if $s=ts^*$ if and only if $s=s^+t$; 
\item \label{i:b5} $s\leq t$ implies $su\leq tu$ and $us\leq ut$ for all $u\in S$;
\item\label{i:b6}  $s\leq t$ implies $s^*\leq t^*$ and $s^+\leq t^+$; 
\item $s\leq t$ implies $s\sim t$;
\item\label{i:b6a} if $s,t\leq u$ for some $u\in S$ then $s\sim t$;
\item\label{i:b8} if $s\sim t$ and $s^*=t^*$ (or $s^+=t^+$) then $s=t$.
\end{enumerate}
\end{lemma}

Restriction semigroups generalize inverse semigroups. Indeed, an inverse semgroup $S$  can be endowed with the structure of a restriction semigroup by putting 
$s^{*} = s^{-1}s$ and $s^{+} = ss^{-1}$ for all $s\in S$.  
Note that, for $S$ inverse, $P(S)$ coincides with the semilattice $E(S)$ of idempotents of $S$. 

As usual, inverse semigroups are considered in the signature $(\cdot\, ,^{-1}\!)$ and inverse monoids in signature $(\cdot\, ,^{-1}\!,1)$. By $FI(X)$ we denote the free inverse monoid over a set $X$.
Since $FR(X)$ is the free $X$-generated restriction monoid, there exist an $X$-canonical morphism (of restriction monoids) $FR(X)\to FI(X)$. 

As already mentioned, monoids and semilattices are basic examples of restriction semigroups, similar to the way in which  groups and semilattices are basic examples of inverse semigroups. It is of fundamental importance that the {\em Ehresmann-Namboripad-Schein theorem} for inverse semigroups admits a natural extension to restriction semigroups, which states that the category of restriction semigroups is isomorphic to the category of {\em inductive categories}. Therefore, restriction semigroups are naturally arising generalizations of of inverse semigroups. However, in general, restriction semigroups are very far from being  inverse, essentially because they do not have an involution.

Let $\sigma$ be the minimum congruence on a restriction semigroup $S$ such that $e \mathrel{\sigma} f$ for all $e,f\in P(S)$. 
Each of the following statements is equivalent to $s\mathrel{\sigma} t$ (see \cite[Lemma 1.2]{CG}): 
\begin{enumerate}
\item there is $e\in P(S)$ such that $es=et$;
\item there is $e\in P(S)$ such that $se=te$. 
\end{enumerate}
Note that $s\sim t$ and, in particular, $s\leq t$, imply $s\mathrel{\sigma} t$, for all $s,t\in S$.

The quotient $S/{\sigma}$ is a reduced restriction semigroup and  is the maximum reduced quotient of $S$.  
By $\sigma^{\natural}\colon S\to S/{\sigma}$ we denote the moprhism induced by~$\sigma$. We also put:
 $$[t]^{\sigma} = \{s\in S \colon \sigma^{\natural}(s)=t\},   \,\, t\in  S/{\sigma}, \,\, \text{ and } \,\  [s]_{\sigma} = \{v\in S\colon v\mathrel{\sigma} s\},\,\, s\in S.$$
 
 In general, if $\rho$ is a congruence on an algebra $S$ and $s\in S$, the congruence class of $s$ is denoted by $[s]_{\rho}$.

A restriction semigroup $S$ is called {\em proper} if the following two conditions hold:
\begin{enumerate}
\item for all $s,t\in S:$  if  $s^*=t^*$    and  $s\mathrel{\sigma} t$ then $s=t$;
\item for all $s,t\in S:$  if  $s^+=t^+$   and  $s\mathrel{\sigma} t$ then $s=t$.
\end{enumerate}

Note that a restriction semigroup is proper if and only if $\sigma$ coincides with the compatibility relation $\sim$. Indeed, if $S$ is proper and $s\mathrel{\sigma} t$, we have $t^+s \mathrel{\sigma} s^+t$ and 
$(t^+s)^+ = (s^+t)^+$. Thus $t^+s = s^+t$ and similarly $st^* = ts^*$ whence $s\sim t$.

It is easy to see and well known (see, e.g., \cite[p. 282]{Szendrei2014}) that in the case where $S$ is an inverse semigroup, $\sigma$ is the minimum group congruence on $S$, and $S$ is proper if and only if it is $E$-unitary.  Thus proper restriction semigroups generalize $E$-unitary inverse semigroups.

A restriction semigroup $S$ is called {\em ample} if for all $s,t,u\in S$:
\begin{equation}\label{eq:ample}
su=tu \Rightarrow su^+=tu^+ \,\,\,\, \text{ and } \,\,\,\, us=ut\Rightarrow u^*s=u^*t.
\end{equation}

If only the first (resp., the second) of these two implications is required to hold, $S$ is called {\em left ample} (resp. {\em right ample}).

A semigroup $S$ is called {\em left} (resp. {\em right}) {\em cancellative} provided that $us=ut$ implies $s=t$ (resp. $su=tu$ implies $s=t$), for all $s,t,u\in S$. It is called {\em cancellative} if it is both left and right cancellative.

If $(A, \leq)$ is a poset and $B\subseteq A$, the {\em order ideal} generated by $B$ is the set
$$
B^{\downarrow} = \{a\in A\colon a\leq b \text{ for some } b\in B\}. 
$$
If $b\in A$ we write $b^{\downarrow}$ for $\{b\}^{\downarrow}$ and call this set the {\em principal order ideal} generated by $b$.

\subsection{Premorphisms from a monoid to a restriction monoid}\label{subs:prem}
Let $M$ be a monoid  and $T$ a restriction monoid. 
\begin{definition} \label{def:prem} {\em A {\em premorphism} from $M$ to $T$ is a map $\varphi\colon M\to T$ such that the following conditions hold:
\begin{enumerate}
\item[(PM1)] $\varphi(1)=1$;
\item[(PM2)]  $\varphi(m)\varphi(n)\leq \varphi(mn)$ for all  $m,n\in M$.
\end{enumerate}}
\end{definition}
Clearly, the following are two equivalent forms of (PM2):
\begin{enumerate}
\item[(PM2$'$)] $\varphi(m)\varphi(n) = \varphi(mn) (\varphi(m)\varphi(n))^*$ for all  $m,n\in M$;
\item[(PM2$''$)] $\varphi(m)\varphi(n) = (\varphi(m)\varphi(n))^+\varphi(mn)$ for all  $m,n\in M$.
\end{enumerate}
Because inverse monoids can be considered as a special case of restriction monoids,  a {\em premorphism from} $M$ {\em to an inverse monoid} arises as a special case of this definition.

If the axiom (PM2) is replaced by a stronger axiom
\begin{enumerate}
\item[(PM3)]  $\varphi(m)\varphi(n)=\varphi(mn)$ for all  $m,n\in M$,
\end{enumerate}
the premorphism $\varphi$ becomes, of course, a {\em monoid morphism}.

\begin{definition}\label{def:strong} {\em A premorphism $\varphi$ from a monoid $M$ to a restriction monoid  is called
 \begin{itemize}
\item {\em left strong} if $\varphi(m)\varphi(n)= \varphi(m)^+ \varphi(mn)$ for all $m,n\in M$;
 \item {\em right strong} if $\varphi(m)\varphi(n)=\varphi(mn)\varphi(n)^*$ for all $m,n\in M$;
 \item {\em strong} if it is both left and right strong.
 \end{itemize}}
 \end{definition}

\begin{remark}{\em Premorphisms from $M$ to the symmetric inverse monoid ${\mathcal{I}}(X)$ have counterparts in terms of {\em partial actions by partial bijections} of $M$ on $X$ (see \cite{Kud}), but in the present paper we choose to adhere to the language of premorphisms. For a comprehensive survey on partial actions, we refer the reader to \cite{D}.}
\end{remark}

\subsection{Proper restriction semigroups}\label{subs:structure_proper}

We now recall the Cornock-Gould structure result on proper restriction semigroups \cite{CG}. We remark that in~\cite{CG} a pair of  partial actions, called a {\em double action}, satisfying certain compatibility conditions,  was considered, and in \cite{Kud} we reformulated this using one partial action by partial bijections. Here we restate the construction of \cite{Kud}  in terms of premorphisms. Let $T$ be a monoid,  $Y=(Y,\wedge)$  a semilattice and assume that we are given a premorphism $\varphi\colon T\to {\mathcal I}(Y)$, $t\mapsto \varphi_t$. We assume that for every $t\in T$ the map $\varphi_t$  satisfies the following axioms: 
\begin{enumerate}
\item[(A)] $\operatorname{\mathrm{dom}} \varphi_t$ and  $\operatorname{\mathrm{ran}}\varphi_t$ are order ideals of $Y$;
\item[(B)] $\varphi_t:\operatorname{\mathrm{dom}}\varphi_t\to \operatorname{\mathrm{ran}}\varphi_t$ is an order-isomorphism;
\item[(C)] $\operatorname{\mathrm{dom}}\varphi_t\neq \varnothing$.
\end{enumerate}

We put:
$$
Y \rtimes_{\varphi} T
=\{(y,t)\in Y\times T\colon  y\in \operatorname{\mathrm{ran}}\varphi_t\}.
$$ 
If $\varphi$ is understood, $Y \rtimes_{\varphi} T$ is sometimes denoted simply by $Y \rtimes T$.

\begin{remark} {\em The notation $Y \rtimes_{\varphi} T$ is new within the given context. In \cite{CG, Kud}  the notation $\mathcal{M}(T,Y)$ is used instead. 
We opted to change the notation because, if $T$ is a group, the construction agrees with the well-known construction of the partial transformation groupoid of a partial action (see, e.g., \cite{MS}), which, in turn, has deep ties with  partial action cross product $C^*$-algebras via the work of Abadie \cite{A}, see \cite{MS}.}
\end{remark}

The multiplication and the unary operations $^*$ and $^+$ on $Y \rtimes_{\varphi} T$  are defined by
\begin{equation}\label{eq:proper1}
(x,s)(y,t)=(\varphi_s(\varphi_s^{-1}(x)\wedge y), st);
\end{equation}
\begin{equation}\label{eq:proper2}
(x,s)^*=(\varphi_s^{-1}(x),1), \,\, (x,s)^+=(x,1).
\end{equation}

Note that  $\varphi_s^{-1}(x)\wedge y \leq \varphi_s^{-1}(x)$ so that  $\varphi_s^{-1}(x)\wedge y \in \operatorname{\mathrm{dom}}\varphi_s$, by axiom (A). Since $y\in \operatorname{\mathrm{ran}}\varphi_t$, we similarly have that $\varphi_s^{-1}(x)\wedge y \in \operatorname{\mathrm{ran}}\varphi_t$. Hence $\varphi_s^{-1}(x)\wedge y \in \operatorname{\mathrm{dom}}\varphi_s\varphi_t$. It now follows from $\varphi_s\varphi_t\leq \varphi_{st}$ that $\varphi_s^{-1}(x)\wedge y \in \operatorname{\mathrm{dom}}\varphi_{st}$, so that \eqref{eq:proper1} makes sense. Furthermore, \eqref{eq:proper2} makes sense since $\varphi_1={\mathrm{id}}_Y$.

With respect to these operations, $Y \rtimes_{\varphi} T$ is a restriction semigroup. It is proper and its semilattice of projections
$$
P(Y \rtimes_{\varphi} T) = \{(y,1)\colon y\in Y\} 
$$
is isomorphic to $Y$ via the map $(y,1)\mapsto y$, $y\in Y$.
The minimum reduced congruence $\sigma$ on $Y \rtimes_{\varphi} T$ is given by
$(x,s)\mathrel {\sigma} (y,t)$ if and only if $s=t$ so that  $(Y \rtimes_{\varphi} T)/\sigma$ is isomorphic to~$T$.

Following the terminology of \cite{Petrich_Reilly79}, it is natural to call $Y \rtimes_{\varphi} T$ the {\em semidirect type product} of $Y$ by $T$ {\em determined by the premorphism} $\varphi$, or just the {\em partial action product} of $Y$ by $T$ where the partial action is determined by the premorphism $\varphi$.

In the reverse direction, let $S$ be a proper restriction semigroup. 

\begin{definition}\label{def:underlying} {\em The {\em underlying premorphism} of $S$ is the premorphism $\varphi\colon S/\sigma \to {\mathcal I}(P(S))$  given, for  $t\in S/\sigma$, by
\begin{equation}\label{eq:und}
\operatorname{\mathrm{dom}}\varphi_t = \{e\in P(S)\colon \text{ there is } s\in [t]^{\sigma} \text{ such that } e\leq s^*\} \,\, \text{ and }
\end{equation}
\begin{equation}\label{eq:und1}
\varphi_t(e) = (se)^+ \, \text{ for any } e\in \operatorname{\mathrm{dom}}\varphi_t \text{ and } s\in [t]^{\sigma}  \text{ such that } e\leq s^*.
\end{equation}} 
\end{definition} 

It follows that, for all $t\in S/\sigma$,
\begin{equation}\label{eq:und2}
\operatorname{\mathrm{ran}}\varphi_t = \{e\in P(S)\colon \text{ there is } s\in [t]^{\sigma} \text{ such that } e\leq s^+\} \,\, \text{ and }
\end{equation}
\begin{equation}\label{eq:und3}
\varphi_t^{-1}(e) = (es)^* \, \text{ for any } e\in \operatorname{\mathrm{ran}}\varphi_t \text{ and } s\in [t]^{\sigma}  \text{ such that } e\leq s^+.
\end{equation}

It can be verified (or see \cite{CG}) that $\varphi$ is well defined and satisfies axioms
 (A), (B) and~(C). Therefore, the partial action product $P(S) \rtimes_{\varphi} S/\sigma$ can be formed and we have:
$$
P(S) \rtimes_{\varphi} S/\sigma = \{(e, t)\in P(S)\times S/\sigma \colon  \text{ there is } s\in [t]^{\sigma}  \text{ such that } e\leq s^+\}.
$$
The operations on $P(S) \rtimes_{\varphi} S/\sigma$ are:
$$
(e,u)(f,v) = (e(sf)^+, uv), \,\, (e,u)^* = ((es)^*, 1), \,\, (e,u)^+ = (e,1), \text{ where } s\in  [u]^{\sigma}.
$$

The following theorem is due to Cornock and Gould \cite{CG}.

\begin{theorem}\label{th:CG} Let $S$ be a proper restriction semigroup. Then the map  $f \colon S\to P(S) \rtimes_{\varphi} S/\sigma$ given by $f(s) = (s^+, \sigma^{\natural}(s))$ is an isomorphism.
\end{theorem}

\begin{remark} {\em The fact that the underlying premorphism of a proper restriction semigroup has its image in an inverse semigroup reveals the connection between proper restriction semigroups and inverse semigroups. The main result of this paper, Theorem \ref{th:main}, makes this connection precise, for classes of proper restriction semigroups considered therein.}
\end{remark}

We conclude this section with the following result.

\begin{theorem}\label{th:ample1} Let $S$ be a proper restriction semigroup. Then $S$ is left ample (resp. right ample or ample) if and only if the monoid $S/\sigma$ is right cancellative (resp. left cancellative or cancellative).
\end{theorem}

\begin{proof} We put $M=S/\sigma$. Due to Theorem \ref{th:CG} we can assume that $S$ is equal to $P(S) \rtimes_{\varphi} M$ where $\varphi$ is the underlying premorphism of $S$.

We first prove the claim that $S$ is left ample if and only if $S/\sigma$ is right cancellative.

We assume that $P(S) \rtimes_{\varphi} M$ is left ample and that  $ac=bc$ in $M$. Let $e\in  \operatorname{\mathrm{ran}}\varphi_a$, $f\in \operatorname{\mathrm{ran}}\varphi_b$ and $g'\in  \operatorname{\mathrm{ran}}\varphi_c$. We put $g=g'\varphi_a^{-1}(e)\varphi_b^{-1}(f)$. Then $g\in \operatorname{\mathrm{ran}}\varphi_c$ (by Axiom~(A)), $g\leq \varphi_a^{-1}(e)\varphi_b^{-1}(f)$ so that $g\in \operatorname{\mathrm{dom}}\varphi_a\cap  \operatorname{\mathrm{dom}}\varphi_b$ (by Axiom (B)). We show that $\varphi_a(g)=\varphi_b(g)$. Since $g\in \operatorname{\mathrm{ran}}\varphi_c\cap  \operatorname{\mathrm{dom}}\varphi_a$, it follows that $\varphi_c^{-1}(g)\in  \operatorname{\mathrm{dom}}\varphi_a\varphi_c$. But $\varphi_a\varphi_c \leq \varphi_{ac}$, since $\varphi$ is a premorphism. Hence $\varphi_c^{-1}(g)\in \operatorname{\mathrm{dom}}\varphi_{ac}$ and 
$\varphi_a(g) = \varphi_a\varphi_c(\varphi_c^{-1}(g)) = \varphi_{ac}(\varphi_c^{-1}(g))$. Similarly, $\varphi_c^{-1}(g)\in \operatorname{\mathrm{dom}}\varphi_{bc}$ and $\varphi_b(g) = \varphi_{bc}(\varphi_c^{-1}(g))$. Hence, in view of $ac=bc$, we have $\varphi_a(g) = \varphi_b(g)$, as desired. Note that $(e,a),(f,b), (g,c) \in P(S) \rtimes_{\varphi} M$ and
\begin{equation}\label{eq:m84}
(e,a)(g,c) = (\varphi_a(g), ac) = (\varphi_b(g), bc) = (f,b)(g,c).
\end{equation}
In addition,
\begin{equation}\label{eq:m85}
(e,a)(g,c)^+ = (e,a)(g,1) = (\varphi_a(g),a)
\end{equation}
and, similarly,
\begin{equation}\label{eq:m86}
(f,b)(g,c)^+ = (f,b)(g,1) = (\varphi_b(g),b).
\end{equation}
It follows from \eqref{eq:m84}, \eqref{eq:m85} and  \eqref{eq:m86} that $a=b$, so that $M$ is right cancellative.

In the reverse direction, we assume that $M$ is right cancellative. Let $(e,a),(f,b), (g,c) \in P(S) \rtimes_{\varphi} M$ be such that $(e,a)(g,c)=(f,b)(g,c)$. We calculate:
$$
(e,a)(g,c) = (\varphi_a(\varphi_a^{-1}(e)g), ac) \, \text{ and } \, (f,b)(g,c) = (\varphi_b(\varphi_b^{-1}(f)g), bc).
$$
The assumption implies that $\varphi_a(\varphi_a^{-1}(e)g)= \varphi_b(\varphi_b^{-1}(f)g)$ and $ac=bc$. Since $M$ is right cancellative, the latter equality yields $a=b$. But then:
$$
(e,a)(g,c)^+ = (e,a)(g,1) = (\varphi_a(\varphi_a^{-1}(e)g), a) = (\varphi_b(\varphi_b^{-1}(f)g), b) = (f,b)(g,c)^+.
$$
It follows that $P(S) \rtimes_{\varphi} M$ is left ample.

We now prove the claim that $S$ is right ample if and only if $S/\sigma$ is left cancellative.

We assume that $P(S) \rtimes_{\varphi} M$ is right ample and that $ca=cb$ in $M$. Note that if $e\in \operatorname{\mathrm{ran}}\varphi_a$ and $f\in \operatorname{\mathrm{ran}}\varphi_b$ then $ef \in \operatorname{\mathrm{ran}}\varphi_a\cap  \operatorname{\mathrm{ran}}\varphi_b$, by Axiom (A). Let $e\in \operatorname{\mathrm{ran}}\varphi_a\cap  \operatorname{\mathrm{ran}}\varphi_b$ and $f\in  \operatorname{\mathrm{ran}}\varphi_c$. Then $(e,a), (e,b), (f,c) \in P(S) \rtimes_{\varphi} M$ and 
\begin{equation}\label{eq:m81}
(f,c)(e,a) = (\varphi_c(\varphi_c^{-1}(f)e), ca) = (\varphi_c(\varphi_c^{-1}(f)e), cb) = (f,c)(e,b).
\end{equation}
In addition,
\begin{equation}\label{eq:m82}
(f,c)^*(e,a)  = (\varphi_c^{-1}(f),1)(e,a) = (\varphi_c^{-1}(f)e,a)
\end{equation}
and, similarly,
\begin{equation}\label{eq:m83}
(f,c)^*(e,b)  = (\varphi_c^{-1}(f),1)(e,b) = (\varphi_c^{-1}(f)e,b).
\end{equation}
From \eqref{eq:m81} we have that $(f,c)^*(e,a) = (f,c)^*(e,b)$. In view of \eqref{eq:m82} and \eqref{eq:m83}, this yields that $a=b$, so that $M$ is left cancellative. 

In the reverse direction, we assume that $M$ is left cancellative. Let $(e,a), (g,b), (f,c) \in P(S) \rtimes_{\varphi} M$ be such that $(f,c)(e,a) = (f,c)(g,b)$. We calculate: 
\begin{equation*}\label{eq:m81a}
(f,c)(e,a) = (\varphi_c(\varphi_c^{-1}(f)e), ca) \, \text{ and } \, (f,c)(g,b) = (\varphi_c(\varphi_c^{-1}(f)g), cb) .
\end{equation*}
It follows that $\varphi_c(\varphi_c^{-1}(f)e) = \varphi_c(\varphi_c^{-1}(f)g)$ and $ca=cb$. The first of these equalities simplifies to $\varphi_c^{-1}(f)e = \varphi_c^{-1}(f)g$ and the second one implies that $a=b$.
Furthermore, similarly as in \eqref{eq:m82} and \eqref{eq:m83}, we have $(f,c)^*(e,a)  = (\varphi_c^{-1}(f)e,a)$ and $(f,c)^*(g,b)  = (\varphi_c^{-1}(f)g,a)$. It follows that $(f,c)^*(e,a)=(f,c)^*(g,b)$. Therefore, $P(S) \rtimes_{\varphi} M$ is right ample.

Combining the two results proved, it follows that $S$ is ample if and only if $M$ is cancellative, which finishes the proof.
\end{proof}

\section{Admissible identities}\label{sect:identities}

 Let $X$ be a set and $X^*$ the free monoid over $X$. The elements of $X^*$ will be called {\em variables}, the empty word will be denoted by $\lambda$. By a $(\cdot\, ,^*,^+,1)$-{\em term}  over $X^*$ we mean an element of the free $X^*$-generated restriction monoid $FR(X^*)$. We denote the identity element of $FR(X^*)$ by $1$ and remark that $\lambda\neq 1$.
 
If $p\in FR(X^*)$, we sometimes write $p$ as $p(u_1,\dots, u_n)$ to indicate that the variables occurring in $p$ are among $u_1,\dots, u_n \in X^*$ (see \cite{BS}).
For a map $f\colon X \to M$, where $M$ is a monoid,  we put $\widehat{f}$ to be the free extension of $f$ to  $X^*$.  

\begin{definition}\label{def:value} {\em Let $M$ be a monoid, $S$ be a restriction monoid and  $f\colon X\to M$, $\gamma\colon M\to S$ be maps. The {\em value} of $p(u_1,\dots, u_n)\in FR(X^*)$ on ${\gamma}$ {\em with respect to} $f$ is defined to be the element
$p(\gamma(\widehat{f}(u_1)),\dots, \gamma(\widehat{f}(u_k)))\in S$ which we denote by $[p(u_1,\dots, u_n)]_{\gamma,f}$.}
\end{definition}  

The following is immediate.

 \begin{lemma}\label{lem:useful26}
Let $M$ be a monoid, $S_1$, $S_2$ be restriction monoids, $\gamma \colon M\to S_1$ be a map and $g\colon S_1\to S_2$ be a morphism.  Then for any $p\in FR(X^*)$ we have:
\begin{equation}\label{eq:value}
[p]_{g\gamma, f} = g([p]_{\gamma, f}).
\end{equation}
\end{lemma}

\begin{definition} {\em A $(\cdot\, ,^*,^+,1)$-{\em identity over}  $X^*$ is an equality of the form
\begin{equation}\label{eq:identity}
p(u_1,\dots, u_n)= q(u_1,\dots, u_n)
\end{equation} where $p(u_1,\dots, u_n), q(u_1,\dots, u_n) \in FR(X^*)$.}
\end{definition} 

In this paper we consider only $(\cdot\, ,^*,^+,1)$-identities over $X^*$, and, for brevity, we refer to them as $(\cdot\, ,^*,^+,1)$-{\em identities}, or simply {\em identities}.

\begin{definition}\label{def:satisfaction} {\em Let $\gamma\colon M\to S$ be a map from a monoid $M$ to a restriction monoid $S$. We say that $\gamma$ {\em satisfies} identity \eqref{eq:identity} if for every  map $f\colon X \to M$ the equality
$$
[p(u_1,\dots, u_k)]_{\gamma,f} = [q(u_1,\dots, u_k)]_{\gamma,f}
$$
holds in $S$. 
We say that $\gamma$ satisfies a set $R$ of identities  if it satisfies each identity in $R$.}
\end{definition}

\begin{definition} \label{def:class} {\em Let $R$ be a set of $(\cdot\, ,^*,^+,1)$-identities over $X^*$, and define $V(R)$ to be the class of all premorphisms satisfying $R$.  A class ${\mathcal{V}}$ of premorphisms from monoids to restriction monoids is an {\em equational class of premorphisms} (or simply an {\em equational class}), if ${\mathcal{V}}=V(R)$ for some set $R$ of $(\cdot\, ,^*,^+,1)$-identities over $X^*$.} 
\end{definition}

\begin{definition}{\em Two sets of $(\cdot\, ,^*,^+,1)$-identities over $X^*$, $R_1$ and $R_2$, are called {\em equivalent} if $V(R_1)=V(R_2)$. }
\end{definition}

The class ${\mathcal{PM}}$ of all premorphisms from  monoids to restriction monoids  is the biggest equaitonal class of premorphisms, i.e., ${\mathcal{PM}}=V(\varnothing)$. Clearly, ${\mathcal{V}}\subseteq {\mathcal{PM}}$ for any equational class of premorphisms ${\mathcal{V}}$.

If $R'$ consists of identities satisfied by {\em all} premorphisms then $V(R) = V(R\cup R') = V(R\setminus R')$ for any set $R$ of identities.  In particular,
${\mathcal{PM}} = V(R')$  for any set~$R'$ of identities satisfied by all premorphisms, for example, for 
\begin{equation*}\label{eq:class_prem}
R' = \{\lambda =1,\, x_1\cdot x_2 = x_1x_2 \cdot (x_1\cdot x_2)^*\}.
\end{equation*}

The following is a central definition.

\begin{definition} \label{def:admissible_identity} {\em 
A $(\cdot\, ,^*,^+,1)$-identity over $X^*$ will be called {\em admissible} if all monoid morphisms satisfy this identity.
}
\end{definition}

We give some examples of sets of admissible identities and equational classes of premorphisms they define.

\begin{example}
 {\em The set $\varnothing$ defines the class ${\mathcal{PM}}$ of all premorphisms. Furthermore, the sets
\begin{equation}\label{eq:class_hom}
R_m = \{x_1\cdot x_2 = x_1x_2\};
\end{equation} 
\begin{equation}\label{eq:class_ls}
R_{ls} = \{x_1\cdot x_2 =  x_1^+\cdot x_1x_2\};
\end{equation}
\begin{equation}\label{eq:class_rs}
R_{rs} = \{x_1\cdot x_2 =  x_1x_2\cdot x_2^*\};
\end{equation}
\begin{equation}\label{eq:class_s}
R_s=R_{ls}\cup R_{rs} 
\end{equation} 
consist of admissible identities.  The class  $V(R_m)$, denoted also by ${\mathcal{M}}$, is the class of all monoid morphisms; the classes $V(R_{ls})$, $V(R_{rs})$ and $V(R_{s})$  are the classes of all left strong premorphisms, right strong premorphisms, and premorphisms, respectively. }
\end{example}

We note that a set $R$ of $(\cdot\, ,^*,^+,1)$-identities over $X^*$ consists of admissible identities if and only if $${\mathcal{M}} \subseteq V(R) \subseteq {\mathcal{PM}}.$$

\begin{remark} {\em We are mostly interested in the cases  $R=\varnothing$, $R=R_m$ and $R=R_s$. Developing the theory for an arbitrary set $R$ of admissible identities, we aim not only to extend the generality, but also to provide a unified and thus conceptually simpler treatment of the special cases of main interest.}
 \end{remark}
 
Observe that, applying \eqref{eq:mov_proj}, any $p\in FR(X^*)$ can be written as
  $p=u_1\cdot \ldots \cdot u_k \cdot e$, where $k\geq 1$, $u_i\in X^*$, $1\leq i\leq k$, and $e\in P(FR(X^*))$.
  It follows that an identity $p=q$, where $p,q\in FR(X^*)$, can be rewritten as
\begin{equation}\label{eq:sat}
 u_1\cdot \ldots \cdot u_k \cdot e = v_1\cdot \ldots \cdot v_n \cdot f,
\end{equation}
where $k,n\geq 1$, $u_1,\dots, u_k, v_1,\dots, v_n\in X^*$ and $e,f\in  P(FR(X^*))$. By definition, a premorphism $\gamma\colon M\to S$  satisfies the identities  $u_1\cdot \ldots \cdot u_k = (u_1\dots u_k) \cdot (u_1\cdot \ldots \cdot u_k)^*$ and $v_1\cdot \ldots \cdot v_n = (v_1\dots v_n) \cdot (v_1\cdot \ldots \cdot v_n)^*$. It follows that $\gamma$ satisfies the identity \eqref{eq:sat} if and only if it satisfies the identity
\begin{equation}\label{eq:sat1}
(u_1\dots u_k) \cdot (u_1\cdot \ldots \cdot u_k)^* \cdot e = (v_1\dots v_n) \cdot (v_1\cdot \ldots \cdot v_n)^* \cdot f.
\end{equation}  
Therefore, for the identity $p=q$ there is an equivalent identity of the form 
\begin{equation}\label{eq:form11} 
u\cdot e = v\cdot f  \text{ where } u,v\in X^* \text{ and }  e,f \in  P(FR(X^*)).
\end{equation}

Assume that  $p=q$ is an admissible identity. Since it is satisfied by any monoid morphism, so does the identity  $u\cdot e = v\cdot f$ just found. In particular, it is satisfied by the morphism ${\mathrm{id}}_{X^*}$. Thus $u=v$ holds in $X^*$.  It follows that the identity $p=q$  is equivalent to an identity of the form 
\begin{equation}\label{eq:form12} 
u\cdot e = u\cdot f  \text{ where } u\in X^* \text{ and }  e,f \in  P(FR(X^*)).
\end{equation} 
We have proved the following.

\begin{proposition} \label{prop:identities10} Let $R$ be a set of admissible identities. Then it is equivalent to another set, $R'$, of admissible identities that contains only identities of the form \eqref{eq:form12}.
\end{proposition}

For example, $V(R_{rs}) = V(R'_{rs})$ where $R'_{rs} = \{x_1x_2\cdot (x_1\cdot x_2)^* =  x_1x_2\cdot x_2^*\}$.

\section{Two-sided expansions of monoids}\label{sect:expansions}

\subsection{The expansions and the universal property} For a monoid $M$, we put $\lfloor M \rfloor=\{\lfloor m\rfloor \colon m\in M\}$. 
The free restriction monoid $FR \lfloor M \rfloor$ is $M$-generated via the map 
$\iota_{FR \lfloor M \rfloor}\colon M \to FR \lfloor M \rfloor$,  $m\mapsto \lfloor m\rfloor$.

\begin{definition}\label{def:expansion_restr} {\em Let $M$ be a monoid.}
\begin{itemize}
\item {\em We define ${\mathcal{FR}}(M)$ to be the quotient of $FR \lfloor M \rfloor$ by the congruence $\widetilde{\delta}$ generated by the relations:
\begin{equation}\label{eq:em19}
\lfloor 1\rfloor = 1;
\end{equation} 
\begin{equation}\label{eq:e1}
\lfloor m\rfloor \lfloor n\rfloor =\lfloor mn\rfloor (\lfloor m\rfloor \lfloor n\rfloor)^*, \,\,\, m,n\in M.
\end{equation}
These relations say that the map $\iota_{{\mathcal{FR}}(M)}\colon M\to {\mathcal{FR}}(M)$, $m\mapsto \lfloor m\rfloor$, is a premorphism. We call ${\mathcal{FR}}(M)$ the  
the {\em  free restriction monoid over} $M$ {\em with respect to premorphisms.}}
\item
{\em  Let $R$ be a set of admissible $(\cdot\, ,^*,^+,1)$-identites over $X^*$. We define ${\mathcal{FR}}_{R}(M)$ to be the quotient of ${\mathcal{FR}}(M)$ by the congruence $\delta_{R}$ generated by the relations:
$$
[p]_{\iota_{{\mathcal{FR}}(M)},f}= [q]_{\iota_{{\mathcal{FR}}(M)},f},$$
where $f$ runs through all maps $X\to M$ and $p=q$ runs through  $R$.
We call ${\mathcal{FR}}_{R}(M)$ the  
the {\em  free restriction monoid over} $M$ {\em with respect to}~$R$.}
\end{itemize}
\end{definition}

Note that ${\mathcal{FR}}_{\varnothing}(M)={\mathcal{FR}}(M)$. The definition of ${\mathcal{FR}}_{R}(M)$ says that the map $\iota_{{\mathcal{FR}}_{R}(M)}\colon M\to {\mathcal{FR}}_{R}(M)$, $m\mapsto \lfloor m\rfloor$, belongs to $V(R)$. Moreover, ${\mathcal{FR}}_{R}(M)$ is $M$-generated via the map $\iota_{{\mathcal{FR}}_{R}(M)}$ and the projection morphism $\delta_R^{\natural} \colon {\mathcal{FR}}(M) \to {\mathcal{FR}}_{R}(M)$ is $M$-canonical. 

Definition \ref{def:expansion_restr} is equivalent to the following universal property of  ${\mathcal{FR}}_{R}(M)$.

\begin{proposition}\label{prop:universal} Let  $R$ be a set of admissible identities and $M$ be a monoid.  Then  for any premorphism $\lambda\colon M\to S$, where $S$ is a  restriction monoid, satisfying $\lambda \in V(R)$,
there is a morphism $\eta\colon  {\mathcal{FR}}_{R}(M) \to S$ of  restriction monoids such that $\lambda = \eta\iota_{{\mathcal{FR}}_{R}(M)}$.
\end{proposition}

\begin{remark}\label{rem:p} {\em Let $\lambda={\mathrm{id}}_M$.   Then $\lambda \in {\mathcal M}\subseteq V(R)$. Proposition \ref{prop:universal} implies that the map $\lfloor m\rfloor \mapsto m$, $m\in M$, extends to an $M$-canonical morphism $p_M\colon {\mathcal{FR}}_{R}(M) \to M$.}
\end{remark}

We record an immediate consequence of Remark \ref{rem:p}.

\begin{corollary}\label{cor:injective}
The premorphism $\iota_{{\mathcal{FR}}_{R}(M)}$ is injective.
\end{corollary} 

\begin{proposition}\label{prop:expansion} Let $R$ be a set of admissible identities and
$\alpha\colon M_1\to M_2$ a monoid morphism. Then there is a morphism $\widetilde{\alpha}\colon {\mathcal{FR}}_{R}(M_1)\to {\mathcal{FR}}_{R}(M_2)$ of  restriction monoids such that the following diagram commutes:
\begin{center}
\begin{tikzcd}[column sep=3.2em, row sep=2.2em]
 {\mathcal{FR}}_{R}(M_1) \arrow[r, "\widetilde{\alpha}"]  \arrow[d, "p_{M_1}"]&   {\mathcal{FR}}_{R}(M_2) \arrow[d, "p_{M_2}"]\\[0.5 cm]
  M_1 \arrow[r, "\alpha"]  & M_2 
  \end{tikzcd}
  \end{center}
\end{proposition}

\begin{proof}   
Let $\alpha' = \iota_{{\mathcal{FR}}_{R}(M_2)}\alpha$. Then $\alpha'\in V(R)$. By Proposition \ref{prop:universal} it extends to a morphism $\widetilde{\alpha}\colon {\mathcal{FR}}_{R}(M_1)\to {\mathcal{FR}}_{R}(M_2)$. That the diagram above commutes is immediate from the construction of $\widetilde{\alpha}$.
\end{proof}

\begin{remark} {\em Using  categorical language, for any set $R$ of admissible identities, we have constructed  a functor, ${\mathcal{FR}}_{R}$, from the category of monoids to the category of restriction monoids, given on objects by $M\mapsto {\mathcal{FR}}_{R}(M)$ and on morphisms by $\alpha\mapsto \tilde{\alpha}$. Moreover, there is a natural transformation, $p$, from ${\mathcal{FR}}_{R}$ to the identity functor on the category of monoids, whose component at a monoid $M$, given by the  map $p_M\colon {\mathcal{FR}}_{R}(M)\to M$, is surjective. It follows that the functor ${\mathcal{FR}}_{R}$ is an {\em expansion} in the sense of \cite{BR}.} 
\end{remark}

\begin{example}\label{ex:expansions}\mbox{}
{\em \begin{itemize}
\item Let $R=R_m$ (see \eqref{eq:class_hom}).
We call ${\mathcal{FR}}_{R_m}(M)$ the {\em free restriction monoid over}~$M$ ({\em with respect to morphisms}). The defining relations for ${\mathcal{FR}}_{R_m}(M)$ are 
\eqref{eq:em19}
and 
\begin{equation}\label{eq:e4}
\lfloor m\rfloor \lfloor n\rfloor =\lfloor mn\rfloor, \,\,\, m,n\in M.
\end{equation}

\item Let $R=R_{ls}$ (see \eqref{eq:class_ls}). We call ${\mathcal{FR}}_{R_{ls}}(M)$ the {\em free restriction monoid over} $M$ {\em with respect to left strong premorphisms}. Its defining relations are~\eqref{eq:em19} and 
 \begin{equation}\label{eq:e2}
\lfloor m\rfloor \lfloor n\rfloor = \lfloor m\rfloor^+\lfloor mn\rfloor, \,\,\, m,n\in M.
\end{equation}

\item Let $R=R_{rs}$ (see \eqref{eq:class_rs}). We call ${\mathcal{FR}}_{R_{rs}}(M)$  the {\em free restriction monoid over} $M$ {\em with respect to right strong premorphisms}. Its defining relations are \eqref{eq:em19} and 
\begin{equation}\label{eq:e3}
\lfloor m\rfloor \lfloor n\rfloor = \lfloor mn\rfloor \lfloor n\rfloor^*, \,\,\, m,n\in M.
\end{equation}

\item Let now $R=R_{s}$ (see \eqref{eq:class_s}). We call ${\mathcal{FR}}_{R_{s}}(M)$  the {\em free restriction monoid over} $M$ {\em with respect to strong premorphisms}. The defining relations for ${\mathcal{FR}}_{R_{s}}(M)$ are \eqref{eq:em19}, \eqref{eq:e2} and \eqref{eq:e3}.
 \end{itemize}
}
\end{example}
 
 \begin{example}\label{ex:free}
 {\em In the case where $M=A^*$ is a free $A$-generated monoid, ${\mathcal{FR}}_{R_m}(A^*)$ is isomorphic to the {\em free restriction monoid} $FR(A)$}. \end{example}

The following diagram illustrates the connection between the defined expansions. Arrows represent components at $M$ (which are all surjective) of the canonical natural transformations between the expansions (looked at as functors).

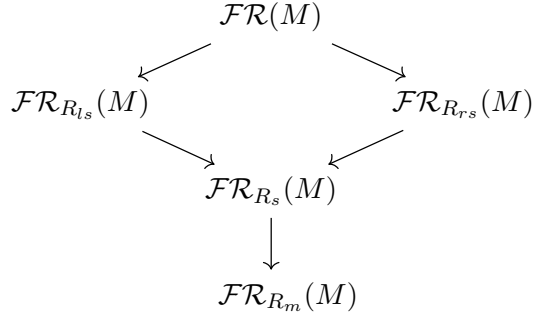
\begin{center}
\begin{tikzcd}[column sep=0.2em, row sep=0.6em]
 & {\mathcal{FR}}(M) \arrow[rd] \arrow[dl] &  \\ [0.5em]
  {\mathcal{FR}}_{R_{ls}}(M) \arrow[rd]  &  & {\mathcal{FR}}_{R_{rs}}(M)  \arrow[ld] \\[0.5em]
  & {\mathcal{FR}}_{R_s}(M) \arrow[d] &\\ [1.1em]
  & \,\,\,\,\,\,\, {\mathcal{FR}}_{R_m}(M) &
  \end{tikzcd}
  \end{center}
  
\section{$F$-restriction monoids and their associated premorphisms} \label{sect:F} 

\subsection{The Munn representation of a restriction semigroup} \label{subs:munn} Recall that the {\em Munn semigroup} $T_Y$ of a semilattice $Y$ is the semigroup of all order-isomorphisms between principal order ideals of $Y$ under composition. This is an inverse semigroup contained in ${\mathcal I}(Y)$. If $Y$ has a top element, $T_Y$ is a monoid whose identity is ${\mathrm{id}}_Y$, the identity map on $Y$. 

The {\em Munn representation} of a restriction semigroup $S$ is a morphism $\alpha\colon S\to T_{P(S)}$, $s\mapsto \alpha_s$, given, for each $s\in S$, by
\begin{equation}\label{eq:munn}
\operatorname{\mathrm{dom}}\alpha_s = (s^*)^{\downarrow} \,  \text{ and } \, \alpha_s(e) = (se)^+, \, e\in \operatorname{\mathrm{dom}}\alpha_s.
\end{equation}
It follows that 
\begin{equation}\label{eq:munn1}
\operatorname{\mathrm{ran}}\alpha_s = (s^+)^{\downarrow} \,  \text{ and } \, \alpha_s^{-1}(e) = (es)^*, \, e\in \operatorname{\mathrm{ran}}\alpha_s.
\end{equation}
If $S$ is an inverse semigroup, $\alpha_s(e) = ses^{-1}$ for all $e\in \operatorname{\mathrm{dom}}\alpha_s$ and $\alpha_s^{-1}(e) = \alpha_{s^{-1}}(e) = s^{-1}es$ for all
$e\in \operatorname{\mathrm{ran}}\alpha_s$.

\subsection{The two premorphisms underlying an $F$-restriction monoid} A restriction semigroup is called $F$-{\em restriction} if every $\sigma$-class has a maximum element. It is easy to check (or see \cite[Lemma 5]{Kud}) that an $F$-restriction semigroup is necessarily a monoid with the identity element being the maximum projection, and is proper. 

Let $S$ be an $F$-restriction monoid. We put $T=S/\sigma$. Inspired by similar considerations for $F$-inverse monoids \cite{KL}, we introduce the map 
\begin{equation}\label{eq:tau1}
\tau_S\colon T\to S \text{ given by } \tau_S(t) ={\mathrm{max}} \, [t]^{\sigma}, \, t\in T.
\end{equation}
\begin{lemma} The map $\tau_S$ is a premorphism. 
\end{lemma}

\begin{proof}  It is immediate that $\tau_S(1)=1$. Let $t_1, t_2\in T$. Since $\tau_S(t_1)\tau_S(t_2)  \mathrel{\sigma} \tau_S(t_1t_2)$, we have $\tau_S(t_1)\tau_S(t_2)  \leq \tau_S(t_1t_2)$.
\end{proof}

Throughout this section, $S$ is an $F$-restriction monoid.  Recall that $\varphi$ is the underlying premorphism of $S$ (see Definition \ref{def:underlying}). To simplify notation, we write $\tau$ for $\tau_S$. 

Due to the fact that $S$ is $F$-restriction, for each $t\in T$, \eqref{eq:und}, \eqref{eq:und1}, \eqref{eq:und2} and \eqref{eq:und3} simplify to:
\begin{equation}\label{eq:auxa22b}
\operatorname{\mathrm{dom}} \varphi_t = (\tau(t)^*)^{\downarrow} \, \text{ and } \, \varphi_t(e) = (\tau(t)e)^+, \,\, e\leq \tau(t)^*;
\end{equation}
\begin{equation}\label{eq:auxa22a}
\operatorname{\mathrm{dom}}\varphi_t^{-1} = \operatorname{\mathrm{ran}} \varphi_t = (\tau(t)^+)^{\downarrow} \,  \text{ and } \, \varphi_t^{-1}(e) = (e\tau(t))^*, \,\, e\leq \tau(t)^+.
\end{equation}

It follows that the image of  $\varphi$ is contained in $T_{P(S)}$ and, moreover, 
$\varphi = \alpha\tau$. 
 
\begin{proposition} \label{prop:agreement} Let $S$ be an $F$-restriction monoid, $T=S/\sigma$, $\varphi$ the underlying premorphism of $S$ and $\tau$ the premorphism given in \eqref{eq:tau1}. Then  $\varphi$ satisfies an admissible identity if and only if $\tau$ satisfies the same identity. 
\end{proposition}

\begin{proof} Let $p\in FR(X^*)$ and  let $g\colon X\to T$ be a map. It follows from \eqref{eq:munn} and \eqref{eq:munn1} that 
\begin{equation}\label{eq:equality11c}
\operatorname{\mathrm{dom}}([p]_{\varphi,g}) = ([p]_{\tau,g}^*)^{\downarrow}, \,\, \,  \, \, \operatorname{\mathrm{ran}}([p]_{\varphi,g}) = ([p]_{\tau,g}^+)^{\downarrow} \,\,\, \text{ and }
\end{equation}
\begin{equation}\label{eq:equality24a}
[p]_{\varphi,g}(e) = ([p]_{\tau,g}e)^+, \, e\leq [p]_{\tau,g}^*; \,\,\, [p]_{\varphi,g}^{-1}(e) = (e[p]_{\tau,g})^*, \, e\leq [p]_{\tau,g}^+.
\end{equation}

Let $p=q$ be an admissible identity. Since $\varphi=\alpha\tau$ and $\alpha$ is a morphism, it folows, in view of Lemma \ref{lem:useful26}, that $[p]_{\tau,g}=[q]_{\tau,g}$ implies $[p]_{\varphi,g}=[q]_{\varphi,g}$.

Suppose that  $[p]_{\varphi,g}=[q]_{\varphi,g}$. Since the identity $p=q$ is satisfied by  ${\mathrm{id}}_{T}$, we have
$$\sigma^{\natural}([p]_{\tau,g}) = [p]_{{\mathrm{id}}_{T},g} = [q]_{{\mathrm{id}}_{T},g} = \sigma^{\natural}([q]_{\tau,g}),$$ 
where for the first and the third equalities we used Lemma \ref{lem:useful26} taking into account that ${\mathrm{id}}_{T} = \sigma^{\natural}\tau$ with $\sigma^{\natural}$ being a morphism. 
Hence $[p]_{\tau,g} \mathrel{\sigma} [q]_{\tau,g}$. Since, in addition, \eqref{eq:equality11c} implies that $[p]_{\tau,g}^*= [q]_{\tau,g}^*$ and since $S$ is proper, we obtain $[p]_{\tau,g} = [q]_{\tau,g}$.
\end{proof} 

Let $S$ be an $F$-restriction monoid. It is called {\em left extra $F$-restriction} (resp.  {\em right extra $F$-restriction} or  {\em extra $F$-restriction}) \cite{Kud} provided that the underlying premorphism $\varphi_S$ is left strong (resp. right strong or  strong). It is called {\em ultra $F$-restriction} \cite{Kud} (or {\em perfect}~\cite{Jones}) if $\varphi_S$ is a morphism.
Proposition \ref{prop:agreement} tells us that all these classes can be equivalently defined by the respective properties of the premorphism $\tau_S$.

\subsection{Universal $F$-restriction property}

\begin{lemma}\label{lem:semi-canonical}
 Let $M$ be a monoid and $R$ a set of admissible identities. 
Then the congruence $\sigma$ on ${\mathcal{FR}}_{R}(M)$ coincides with ${\mathrm{ker}}(p_M)$. Consequently, ${\mathcal{FR}}_{R}(M)/\sigma\simeq M$.
\end{lemma}

\begin{proof} 
Let $a,b\in {\mathcal{FR}}_{R}(M)$ be such that $p_M(a) = p_M(b)$. Applying \eqref{eq:mov_proj}, we  rewrite the elements $a$ and $b$ as $a =  \lfloor u\rfloor  e$ and $b=\lfloor v \rfloor f$ where $u,v\in M$ and $e,f\in P({\mathcal{FR}}_{R}(M))$. Then $u=p_M(a) = p_M(b) = v$. Hence $\lfloor u\rfloor  ef \leq a,b$ which implies $a\mathrel{\sigma} b$.
\end{proof}

We can now state the universal $F$-restriction property of ${\mathcal{FR}}_{R}(M)$. Corollary~\ref{cor:group} shows that it extends a result by Szendrei \cite{Szendrei}.

\begin{proposition}\label{prop:univ_F}  Let $R$ be a set of admissible identities and $M$ be a monoid. 
\begin{enumerate} 
\item The restriction monoid ${\mathcal{FR}}_{R}(M)$ is $F$-restriction and $\tau_{{\mathcal{FR}}_{R}(M)} = \iota_{{\mathcal{FR}}_{R}(M)}$. In particular, ${\mathcal{FR}}_{R}(M)$ is proper.
\item Let $S$ be an $F$-restriction  monoid such that $S/\sigma\simeq M$ and assume that $\tau_S\in V(R)$. Then there is a morphism of  restriction monoids $\eta\colon {\mathcal{FR}}_{R}(M) \to S$  such that $\eta \lfloor m \rfloor = \tau_S(m)$ for all $m\in M$ and $p_M=\sigma_S^{\natural}\eta$.
\end{enumerate}
\end{proposition}

\begin{proof} (1)  This is immediate.

(2) This follows from Proposition~\ref{prop:universal}. 
\end{proof}
\begin{remark}{\em Proposition \ref{prop:univ_F} shows that the functor ${\mathcal{FR}}_{R}$ from the category of monoids to the category of $F$-restriction monoids whose underlying premorphism belongs to $V(R)$ is a left adjoint to the functor in the reverse direction that takes $S$ to $S/\sigma$.}
\end{remark}

The following result is a consequence of Theorem \ref{th:ample1} and Proposition \ref{prop:univ_F}(1).

\begin{corollary}\label{cor:ample}
Let $M$ be a monoid and $R$ a set of admissible identities. Then $M$ is left cancellative (resp. right cancellative, cancellative) if and only if ${\mathcal{FR}}_{R}(M)$ is right ample (reps. left ample, ample).
\end{corollary}

\section{The main result} \label{sect:main}
\subsection{Free inverse monoids over $M$} 
 For a monoid $M$, we put $[M]=\{[m]\colon m\in M\}$. The free inverse monoid $FI[M]$ is $M$-generated via the map $\iota_{FI[M]}\colon M\to FI[M]$, $m\mapsto [m]$.

The following definition is analogous to Definition \ref{def:expansion_restr}.

\begin{definition}\label{def:expansion_inv} Let $M$ be a monoid.
\begin{itemize}
\item {\em We define ${\mathcal{FI}}(M)$ to be the 
quotient of $FI[M]$ by the congruence $\widetilde{\gamma}$ generated by the relations:
\begin{equation}\label{eq:em19inv}
[1]= 1;
\end{equation} 
\begin{equation}\label{eq:e1inv}
[m] [n] =[mn] ([m] [n])^*, \,\,\, m,n\in M.
\end{equation}
The defining relations say that the map $\iota_{{\mathcal{FI}}(M)}\colon M\to {\mathcal{FI}}(M)$, $m\mapsto [m]$, is a premorphism.
We call ${\mathcal{FI}}(M)$ the  {\em  free inverse monoid over} $M$ {\em with respect to premorphisms.}}

\item  {\em Let $R$ be a set of admissible identities and $M$ be a monoid. We define ${\mathcal{FI}}_R(M)$ to be the 
quotient of ${\mathcal{FI}}(M)$ by the congruence $\gamma_R$ generated by the relations:
$$
[p]_{\iota_{{\mathcal{FI}}(M)},f}= [q]_{\iota_{{\mathcal{FI}}(M)},f},$$
where $f$ runs through all maps $X\to M$ and $p=q$ runs through $R$.
We call ${\mathcal{FI}}_{R}(M)$ the  {\em  free inverse monoid over} $M$ {\em with respect to}~$R$.} 
\end{itemize}
\end{definition}

Standard arguments show that an analogue of Proposition \ref{prop:universal} holds true, that is, the map $\iota_{{\mathcal{FI}}_{R}(M)}\colon M\to {\mathcal{FI}}_R(M)$, $m\mapsto [m]$, is a premorphism satisfying identities $R$ that is universal amongst all such premorphisms from $M$ to inverse monoids. 

Observe that ${\mathcal{FI}}(M) = {\mathcal{FI}}_{\varnothing}(M)$. We highlight several more examples of interest:
\begin{itemize}
\item ${\mathcal{FI}}_{R_m}(M)$, the {\em free inverse monoid over} $M$ ({\em with respect to morphisms});
\item ${\mathcal{FI}}_{R_{ls}}(M)$, the {\em free inverse monoid over} $M$ {\em with respect to left strong premorphisms}.
\item ${\mathcal{FI}}_{R_{rs}}(M)$, the {\em free inverse monoid over} $M$ {\em with respect to right strong premorphisms}.
\item ${\mathcal{FI}}_{R_{s}}(M)$, the {\em free inverse monoid over} $M$ {\em with respect to strong premorphisms}.
\end{itemize}
 
\begin{example}
{\em ${\mathcal{FI}}_{R_m}(A^*)$ is isomorphic to the {\em free inverse  monoid} $FI(A)$ over~$A$.} 
\end{example}

\begin{remark} {\em The premorphism $\iota_{{\mathcal{FI}}_R(M)}$ is not injective in general. For example, an easy calculation shows that if $M$ is a rectangular band with the identity element added then  ${\mathcal{FI}}_{R_m}(M)$ is two-element. 
}
\end{remark}
\begin{remark}
{\em It is immediate by the definition that if $M$ is an inverse monoid,  ${\mathcal{FI}}_{R_m}(M)$ is isomorhic to $M$. Thus one can not expect ${\mathcal{FI}}_R(M)$ to be in general $F$-inverse or \mbox{$E$-unitary}. There is no analogue of Proposition \ref{prop:expansion} either. (The difference with the respective property of ${\mathcal{FR}}_R(M)$ is essentially because the maximum reduced image of the {\em inverse monoid} ${\mathcal{FI}}_R(M)$ is a {\em group}, and thus can not be just monoid $M$, as it is for ${\mathcal{FR}}_R(M)$.) In particular, the functor ${\mathcal{FI}}_R$ {\em is not} an expansion}.
\end{remark}

\subsection{The partial action product $E({\mathcal{FI}}_R(M)) \rtimes M$ and the main result} \label{subs:partial_prod}
Let  $R$ be a set of admissible identities.
Since ${\mathcal{FI}}_R(M)$ is a restriction monoid (with $a^*=a^{-1}a$ and $a^+=aa^{-1}$) and $\iota_{{\mathcal{FI}}_R(M)}\in V(R)$,  Proposition \ref{prop:universal} implies that there is an $M$-canonical morphism 
 \begin{equation}\label{eq:psir}
\psi_R\colon {\mathcal{FR}}_{R}(M) \to  {\mathcal{FI}}_{R}(M).
\end{equation}

Note that $\psi_R$ does not need to be surjective.  Wondering if there is a closer connection between ${\mathcal{FR}}_{R}(M)$ and ${\mathcal{FI}}_R(M)$ we make the following crucial observations.

Let $\alpha\colon {\mathcal{FI}}_R(M) \to {\mathcal{I}}(E({\mathcal{FI}}_R(M)))$  be the Munn representation of ${\mathcal{FI}}_R(M)$ (see Subsection \ref{subs:munn}). We put $\beta = \alpha\iota_{{\mathcal{FI}}_R(M)}\colon M \to {\mathcal{I}}(E({\mathcal{FI}}_R(M)))$, $m\mapsto \beta_m$.
We have: 
$$\operatorname{\mathrm{dom}} \beta_m = ([m]^{*})^{\downarrow}, \,\, \operatorname{\mathrm{ran}} \beta_m = ([m]^{+})^{\downarrow};
$$
$$
\beta_m (e) = [m]e[m]^{-1} = ([m]e)^{+}, \,\, e\leq [m]^*;
$$
$$
\beta_m^{-1}(e) = [m]^{-1}e[m] = (e[m])^{*}, \,\, e\leq [m]^+.
$$
\begin{lemma} The map $\beta$ is a premorphism and $\beta\in V(R)$. Moreover, $\beta$  satisfies conditions (A), (B), (C) of Subsection~\ref{subs:structure_proper} for forming the partial action product $E({\mathcal{FI}}_R(M)) \rtimes_{\beta} M$. 
\end{lemma}

\begin{proof} The first assertion holds in view of Lemma \ref{lem:useful26} because $\beta = \alpha\iota_{{\mathcal{FI}}_R(M)}$ with $\alpha$ being a morphism and $\iota_{{\mathcal{FI}}_R(M)}\in V(R)$.
The second assertion is immediate. 
\end{proof}

In the sequel, to prevent overload of the notation, we suppress the index $\beta$ of $E({\mathcal{FI}}_R(M)) \rtimes_{\beta} M$.
We have:
$$E({\mathcal{FI}}_R(M)) \rtimes M = \{(e,m)\in E({\mathcal{FI}}_R(M))\times M \colon e \leq [m]^+\}.$$
The identity element of  $E({\mathcal{FI}}_R(M)) \rtimes M$ is $(1,1)$ and the rest of the operations are: 
\begin{equation}\label{eq:prod_operations}
(e,m)(f,n) = (e([m]f)^+, mn), \,\, (e,m)^* = ((e[m])^*, 1), \,\, (e,m)^+ = (e,1).
\end{equation}

\begin{proposition}\label{prop:action} \mbox{}
\begin{enumerate}
\item Let $(e,m), (f,n) \in E({\mathcal{FI}}_R(M)) \rtimes M$. Then  $(e,m) \mathrel{\sigma} (f,n)$ if and only if $m=n$. 
Consequently, $E({\mathcal{FI}}_R(M)) \rtimes M$ is $F$-restriction with ${\mathrm{max}} [(e,m)]_{\sigma} = ([m]^+,m)$.
\item The map $\tau_R\colon M \to E({\mathcal{FI}}_R(M)) \rtimes M$, given by $\tau_R(m) = ([m]^{+}, m)$, is a premorphism and $\tau_R\in V(R)$.
\end{enumerate}
\end{proposition}

\begin{proof} 
(1) Since $E({\mathcal{FI}}_R(M)) \rtimes M$ is proper, $\sigma$ coincides with the compatibility relation~$\sim$.
We thus have $(e,m) \mathrel{\sigma} (f,n)$ if and only if 
\begin{equation}\label{eq:comp_aux}
(e,m)(f,n)^* = (f,n)(e,m)^* \text{ and } (f,n)^+(e,m) = (e,m)^+(f,n).
\end{equation}
The left-hand side of the first of these equalities equals
\begin{equation*}
(e,m)(f,n)^*  = (e,m)((f[n])^*,1) = (e([m](f[n])^*)^+, m).
\end{equation*}
Similarly, the right-hand side equals $(f,n)(e,m)^* = (f([n](e[m])^*)^+, n)$. Thus if the first equality of \eqref{eq:comp_aux} holds then $m=n$. Conversely, if $m=n$ then $f([n](e[m])^*)^+  = f([m](e[m])^*)^+ = f(e[m])^* = (e[m]f)^*$ and, similarly, $e([m](f[n])^*)^+=
(e[n]f)^* = (e[m]f)^*$, so that the first equality of \eqref{eq:comp_aux} holds. The second equality of \eqref{eq:comp_aux} is equivalent to $(f,1)(e,m) = (e,1)(f,n)$ which is equivalent to $(ef, m) = (ef, n)$ which holds if and only if $m=n$. The description of $\sigma$ follows. It is now clear that $([m]^+,m) = {\mathrm{max}} [(e,m)]_{\sigma}$.

(2) Part (1) implies that $\tau_R = \tau_{E({\mathcal{FI}}_R(M)) \rtimes M}$ (see \eqref{eq:tau1}). The claim that $\tau_R\in V(R)$ now follows from Proposition \ref{prop:agreement}.
\end{proof}

 Applying Proposition \ref{prop:universal}, it follows that there is a  morphism 
\begin{equation}\label{eq:eta1}
\eta_R\colon {\mathcal{FR}}_{R}(M)\to E({\mathcal{FI}}_R(M)) \rtimes M.
\end{equation}
of restriction monoids such that $\tau_R = \eta_R\iota_{{\mathcal{FR}}_{R}(M)}$.
We are coming to our main result.

\begin{theorem}\label{th:main} Let $M$ be a monoid and  $R$ a set of admissible identities. Then $\eta_R$  is an isomorphism of restriction monoids.
\end{theorem}

It is enough to prove that $\eta_R$ is surjective (i.e., that $E({\mathcal{FI}}_R(M)) \rtimes M$ is $M$-generated via~$\tau_R$) and  that there is an $M$-canonical morphism $\Psi_R\colon E({\mathcal{FI}}_R(M)) \rtimes M\to {\mathcal{FR}}_{R}(M)$. 

\section{Proof of the main result}  \label{sect:proof}
\subsection{The map $u\mapsto D_u$ from $FI(X)$ to $P(FR(X))$} Let $X$ be a set. We aim to construct a map from $FI(X)$ to $P(FR(X))$. As a first step, we construct a map from the {\em free involutive monoid} $F_{inv}(X)$ to $P(FR(X))$.  Recall that elements of $F_{inv}(X)$ are words over $X\cup X^{-1}$.  We provide a recursive construction as to how to associate $u\in F_{inv}(X)$ with a projection $D_u\in FR(X)$. 

If $u=1$, we put $D_u=1$. 
 Let now $|u|=n\geq 1$ and assume that for words $v$ with $|v| < n$ the elements $D_v$ are already defined. We then put:
\vspace{0.1cm}
\begin{equation}\label{eq:du}
D_u=\left\lbrace\begin{array}{ll} (D_v x)^{*}, & \text{if } u=vx, \, x\in X,\\
(x D_v)^{+}, & \text{if } u=vx^{-1}, \, x\in X. \end{array}\right.
\end{equation}
In particular, for $x\in X$, we have $D_{x}=x ^{*}$ and $D_{x^{-1}}=x^{+}$. 
\begin{remark} {\em The map $u\mapsto D_u$ is inspired by the map $\theta'\colon FG(X)\to P(M)$, where $FG(X)$ is the free group over $X$ and $M$  a restriction monoid, from \cite{FGG}. In fact, we use the same construction as in \cite{FGG}, but with different domain and range.}
\end{remark}

The following is immediate from \eqref{eq:du}.

\begin{lemma} \label{lem:aux10} Let $u,v,w \in F_{inv}(X)$.
 If $D_u=D_v$ then $D_{uw} = D_{vw}$.
\end{lemma}

Observe that, by \eqref{eq:consequences}, $D_{vxy} = ((D_vx)^*y)^* = (D_v xy)^*$. By induction, this and its dual equality, involving the operation $^+$, lead to the following equalities that hold for all $w\in X^*$:
\begin{equation}\label{eq:aux11a} 
D_{v w} = (D_v w)^* \text{ and } D_{v w^{-1}}= (w D_v)^+.
\end{equation}

Recall \cite{Lawson} that the free inverse monoid $FI(X)$ can be realized as the quotient of $F_{inv}(X)$ by the congruence $\rho $ generated by pairs
$(x x^{-1}x,  x)$,  $x\in X\cup X^{-1}$, and $(xx^{-1}yy^{-1}, yy^{-1}xx^{-1})$,  $x,y \in X\cup X^{-1}$.

\begin{lemma}\label{lem:new13} Let $u,v\in F_{inv}(X)$. Then:
\begin{enumerate}
\item $D_{uxx^{-1}xv}=D_{uxv}$ for all $x\in X\cup X^{-1}$;
\item $D_{uxx^{-1}yy^{-1}v}=D_{uyy^{-1}xx^{-1}v}$ for all $x,y\in X\cup X^{-1}$.
\end{enumerate}
\end{lemma}

\begin{proof} Due to Lemma \ref{lem:aux10} we can assume that $v=1$.

(1)  Let $x\in X$. Using \eqref{eq:fu}, we calculate:
\begin{equation*}
D_{uxx^{-1}x}= ((x (D_u x)^*)^+x)^* = (D_u x^+x)^* = (D_u x)^* = D_{ux}.
\end{equation*}
Similarly one shows that $D_{ux^{-1}xx^{-1}} = D_{ux^{-1}}$.

(2) 
Let $x,y \in X$. Applying \eqref{eq:fu}, we have $D_{uxx^{-1}} = D_u x^+$ and $D_{ux^{-1}x} = D_u x^*$.
Thus:
$$
\begin{array}{ccccccc}
D_{uxx^{-1}yy^{-1}}  & = &  D_u x^+y^+  &  = & D_uy^+x^+  &=  & D_{uyy^{-1}xx^{-1}};\\
D_{uxx^{-1}y^{-1}y} & = & D_u x^+y^* & = & D_u y^*x^+ & =  & D_{uy^{-1}yxx^{-1}}; \\
D_{ux^{-1}xy^{-1}y} & = & D_u x^*y^* & = & D_u y^*x^* & =  & D_{uy^{-1}yx^{-1}x}.
\end{array}
$$\end{proof}

\begin{corollary}\label{cor:equiv1} Let $u,v\in F_{inv}(X)$. If $u \mathrel{\rho} v$ then $D_u = D_v$.
\end{corollary}

Let $u\in FI(X)$. Due to Corollary \ref{cor:equiv1} there is a well-defined element $D_{u} \in P(FR(X))$ that is given by \eqref{eq:du}. 

\subsection{Some properties of elements $D_u$} We now establish some properties of elements $D_u$ that will be needed in the sequel.
Applying \eqref{eq:fu} and \eqref{eq:aux11a},  it follows that for any $e\in E(FI(X))$ and $u\in X^*$, we have: 
\begin{equation}\label{eq:aux_14a}
D_{u^*e}=D_{eu^*} = D_{e}u^* \text{ and } D_{u^+e}=D_{eu^+} = D_{e}u^+.
\end{equation}
Observe that $D_u=D_{u^*}$ and $D_u=D_{u^+}$ for all $u\in X^*$. Applying Lemma \ref{lem:aux10} and \eqref{eq:aux_14a}, it follows that, for all $u\in X^*$ and $e\in E(FI(X))$ we have:
\begin{equation}\label{eq:aux_14a1}
D_{ue} = D_{e}u^* \text{ and } D_{u^{-1}e} = D_{e}u^+.
\end{equation}

\begin{lemma}\label{lem:crucial} Let $u \in FI(X)$. Then $D_{u} = D_{u^*}$.
\end{lemma}

\begin{proof}  
We apply induction on $|u|$. The case $u=1$ is trivial. 
Let $n\geq 1$ and assume that in the case where $|u|<n$ the claim is proved.  Let $u=vx$,  $x\in X\cup X^{-1}$. 
If $x\in X$ we have:
\begin{align*}
D_{u^*} = D_{x^{-1}v^*x} & = (D_{x^{-1}v^*}x)^* & (\text{by } \eqref{eq:du})\\
& = (D_{v^*}x^+x)^* & (\text{by } \eqref{eq:aux_14a1})\\
& = (D_vx^+x)^* & (\text{by the induction hypothesis})\\
& = (D_vx)^* = D_{vx} = D_u. & (\text{applying } \eqref{eq:du})
\end{align*}
The case where $x\in X^{-1}$ is treated similarly.
\end{proof}

\begin{lemma}\label{lem:idempotents} Let $v\in FI(X)$ and $e\in E(FI(X))$. Then $D_{v}D_e=D_{ve}$.
\end{lemma}

\begin{proof} Let $|e|$ denote the minimal length of a word $u$ over $X\cup X^{-1}$ such that $e=u^{-1}u$ in $FI(X)$. We argue by induction on $n=|e|$. For $n=0$ the claim clearly holds. 

Let $n\geq 1$ and assume that if $|e|<n$, the claim holds. We prove the claim for $|e|=n$. Let $e=u^{-1}u$. 
Suppose that $u=wx$ where $x\in X$.  Then, in view of Lemma \ref{lem:crucial} and \eqref{eq:frequent}, we have:
$$
D_vD_{(wx)^*} = D_vD_{wx} = D_v (D_wx)^* = (D_wx)^*D_v = (D_wxD_v)^*.
$$
On the other hand,
\begin{align*}
D_{v(wx)^*} & = D_{vx^{-1}w^*x} = (D_{vx^{-1}w^*}x)^* & (\text{by } \eqref{eq:du})\\
& = (D_{vx^{-1}}D_{w^*}x)^* & (\text{by the induction hypothesis})\\
& = ((xD_{v})^+D_{w^*}x)^* = (D_{w^*}(xD_{v})^+x)^* & (\text{applying } \eqref{eq:du})\\
& = (D_{w}xD_{v})^*, & (\text{by } \eqref{eq:mov_proj} \text{ and Lemma } \ref{lem:crucial})
\end{align*}
so that $D_vD_{(wx)^*} = D_{v(wx)^*}$. The case where $u=wx^{-1}$, $x\in X$, is treated similarly.
\end{proof}

Let $\widetilde{\psi}\colon FR(X)\to FI(X)$ be the $X$-canonical projection map.

\begin{lemma}\label{lem:d_u}
Let $u\in FI(X)$. Then $\widetilde{\psi}(D_u)=u^{*}$.
\end{lemma}

\begin{proof} We apply induction on $|u|$. If $|u|=0$, the statement is obvious. We assume that $\widetilde{\psi}(D_v)=v^{*}$ and let $x\in X$. If $u=vx$ we have $\widetilde{\psi}(D_u)=(v^*x)^{*} = x^{-1} v^{-1}v x = (vx)^{*}$, and if $u=vx^{-1}$ we have $\widetilde{\psi}(D_u) = (x v^*)^{+} = xv^{-1}vx^{-1} = (vx^{-1})^{*}$, as needed.
\end{proof}

It is useful to think of $\widetilde{\psi}(D_u)$ as of $(\cdot\, , ^*, ^+,1)$-decomposition of $u^*\in FI(X)$  over $X$. In particular, we have shown that such a decomposition always exists. 

\begin{example} {\em For $u=xy^{-1}yz$ we have $D_u =  (((yx^{*})^{+}y)^{*})z)^{*} \in P(FR(X))$ and thus $u^*= \widetilde{\psi}(D_u) =  (((yx^{*})^{+}y)^{*})z)^{*}\in E(FI(X))$.}
\end{example}

\subsection{The induced map $u\mapsto D_u$ from ${\mathcal{FI}}(M)$ to $P({\mathcal{FR}}(M))$}
Recall that  ${\mathcal{FI}}(M)$ is a quotient of $FI [M]$ by  the congruence $\widetilde{\gamma}$ (see Definition \ref{def:expansion_restr}). Likewise, ${\mathcal{FR}}(M)$ is a quotient of $FR\lfloor M \rfloor$ by the congruence $\widetilde{\delta}$ (see Definition \ref{def:expansion_inv}).  

\begin{proposition}\label{prop:crucial}
Let $u,v \in FI [M]$. If $u \mathrel{\widetilde{\gamma}} v$ then $D_u \mathrel{\widetilde{\delta}} D_v$.
\end{proposition}

\begin{proof} It sufficies to assume that $u$ and $v$ are ${\widetilde{\gamma}}$-equivalent in one step, that is, via a relation $[m][n]=[mn]([m][n])^{*}$ or its inverse relation $([m][n])^{-1}=([mn]([m][n])^{*})^{-1}$.   
We consider two cases.

{\em Case 1.} Let $u=p[m][n]q$ and $v=p[mn]([m][n])^{*}q$. 
We show that 
\begin{equation} \label{eq:claim} D_{u} \mathrel{{\widetilde{\delta}}} D_{v} \text{ and }  \, D_{u^{-1}} \mathrel{{\widetilde{\delta}}} D_{v^{-1}}.
\end{equation}
 We first show that $D_{u} \mathrel{\widetilde{\delta}} D_{v}$. Applying \eqref{eq:frequent}, we have:
$$
D_{p[m][n]}  = (D_p \lfloor m\rfloor \lfloor n\rfloor)^* \mathrel{\widetilde{\delta}} (D_p\lfloor mn \rfloor(\lfloor m\rfloor\lfloor n\rfloor)^*)^* = (D_p\lfloor mn\rfloor)^*(\lfloor m\rfloor\lfloor n\rfloor)^*.
$$
On the other hand, applying Lemma \ref{lem:idempotents} and \eqref{eq:aux_14a} and \eqref{eq:aux11a}, we have:
\begin{equation*}
D_{p[mn]([m][n])^{*}} = D_{p[mn]}D_{([m][n])^*} = D_{p[mn]}(\lfloor m\rfloor\lfloor n\rfloor)^* = (D_p\lfloor mn\rfloor)^*(\lfloor m\rfloor\lfloor n\rfloor)^*,
\end{equation*}
so that $D_{p[m][n]}  \mathrel{\widetilde{\delta}} D_{p[mn]([m][n])^{*}}$ and thus, by Lemma \ref{lem:aux10}, $D_{u} \mathrel{\widetilde{\delta}} D_{v}$. 

We now show that $D_{u^{-1}}  \mathrel{\widetilde{\delta}} D_{v^{-1}}$. We have $u^{-1}=q^{-1}[n]^{-1}[m]^{-1}p^{-1}$ and $v^{-1} =  q^{-1}([m][n])^{*}[mn]^{-1} p^{-1}$ which is convenient to rewrite as
$v^{-1} = q^{-1}[mn]^{-1}([m][n])^{+} p^{-1}$. Similarly as above, one shows that each of the elements   
$D_{q^{-1}[n]^{-1}[m]^{-1}}$ and $D_{q^{-1}[mn]^{-1}([m][n])^{+}}$ is ${\widetilde{\delta}}$-related to  $(\lfloor m \rfloor \lfloor n \rfloor)^+ (\lfloor mn \rfloor D_{q^{-1}})^+$. It follows that $D_{u^{-1}} \mathrel{\widetilde{\delta}} D_{v^{-1}}$. 

{\em Case 2.} Let $u=p([m][n])^{-1}q$ and $v=p([mn]([m][n])^{*})^{-1}q$. Then
$u^{-1} = q^{-1} [m][n] p^{-1}$ and $v^{-1} = q^{-1}[mn]([m][n])^{*}p^{-1}$, and this case is reduced to the previous one. 
\end{proof}

Hence there is a well-defined map ${\mathcal{FI}}(M)\to P({\mathcal{FR}}(M))$, given by $[e]_{\widetilde{\gamma}}\mapsto [D_e]_{\widetilde{\delta}}$.

\subsection{Proof of Theorem \ref{th:main}: the case where $R=\varnothing$}
We are now ready to complete the proof of Theorem \ref{th:main} in the special case where $R=\varnothing$. We denote $\eta_{\varnothing}$ and $\tau_{\varnothing}$  by $\eta$ and $\tau$, respectively. We aim to prove that
\begin{equation*}
\eta\colon {\mathcal{FR}}(M) \to E({\mathcal{FI}}(M)) \rtimes M
\end{equation*} 
is an isomorphism.

\begin{proposition} \label{prop:generation} The morphism $\eta$  is surjective.  That is,
 $E({\mathcal{FI}}(M)) \rtimes M$ is generated (as a restriction monoid) by the set $\tau(M) = \{([m]^{+}, m)\colon m\in M\}$. 
\end{proposition}

\begin{proof} Let $(e, m) \in E({\mathcal{FI}}(M)) \rtimes M$. We show that $(e,m)\in \langle \tau(M)\rangle$ with the help of induction on the smallest length $n$ of a word $u$ over $[M]\cup [M]^{-1}$ such that the value of $u^*=u^{-1}u$ in ${\mathcal{FI}}(M)$ is $e$. If $n=0$, we have $e=1$, and thus $(e,m)=(1,1)$ which is the nullary operation in $\langle \tau(M)\rangle$. Hence the basis of the induction is established.
Assume that $n\geq 1$ and that the claim is proved for $e$ which can be written as $e=u^*$ where $|u|\leq n-1$. Let $e=v^*$ with $|v|=n$. If $v=u[a]$,  $a\in M$,  applying \eqref{eq:prod_operations} we have:
$$(e,1) = ((u[a])^{*},1)= (u^*[a]^{+}, a)^* = ((u^*,1)([a]^+,a))^* \in  \langle \tau(M) \rangle.$$
If $v=u[a]^{-1}$, $a\in M$, we similarly have:
$$
(e,1)= ((u[a]^{-1})^{*}, 1) = (([a]u^*)^{+}, 1) = (([a]^{+}, a)(u^*, 1))^+ \in \langle \tau(M) \rangle.
$$
Now, $(e,m) = (e,1) ([m]^{+}, m) \in \langle \tau(M) \rangle$ for any $m$ satisfying \mbox{$e\leq [m]^{+}$.}
\end{proof}

We define a map 
\begin{equation}\label{eq:psi}
\Psi \colon E({\mathcal{FI}}(M))  \rtimes M \to {\mathcal{FR}}(M), \,\, \Psi (e,m) = D_e \lfloor m\rfloor.
\end{equation} 
We verify that $\Psi$ is a morphism of restriction monoids. Clearly, it respects the identity element.
Verifications that $\Psi$ respects the unary operations $^*$ and $^+$  amount to showing the equalities $D_{(e[m])^{*}} = (D_e\lfloor m\rfloor)^*$ and $D_e=(D_e\lfloor m\rfloor)^+$, respectively, for any $m\in M$ and any $e\in E({\mathcal{FI}}(M))$ satisfying $e\leq [m]^{+}$. They  follow applying \eqref{eq:du} and Lemma \ref{lem:crucial}:  
$$(D_e\lfloor m\rfloor)^* = D_{e[m]} = D_{(e[m])^{*}}, \, (D_e\lfloor m\rfloor)^+ = (D_e\lfloor m\rfloor^+)^+ = D_e\lfloor m\rfloor^+ = D_{e[m]^{+}} = D_e.$$

It remains to verify that $\Psi$ preserves the multiplication. We have: 
\begin{align*}
\Psi((e,m)(f,n)) & = \Psi(e([m]f)^{+}, mn) & (\text{by } \eqref{eq:prod_operations}) \\ 
 & = \,\, D_{e([m]f)^{+}}\lfloor mn\rfloor & (\text{by the definition of }\Psi)\\
 & = \,\, D_{e}D_{([m]f)^{+}}\lfloor mn\rfloor & (\text{by Lemma \ref{lem:idempotents}})\\
 & = \,\, D_{e}D_{(f[m]^{-1})^{*}}\lfloor mn\rfloor & (\text{since } ([m]f)^{+}=(f[m]^{-1})^{*})\\
 & = \,\, D_{e}D_{f[m]^{-1}}\lfloor mn\rfloor & (\text{by Lemma \ref{lem:crucial}})\\
& = \,\,  D_{e}(\lfloor m\rfloor D_f)^+\lfloor mn\rfloor.  & (\text{by  }\eqref{eq:du})
\end{align*}

On the other hand, $$\Psi((e,m))\Psi((f,n)) = D_e \lfloor m\rfloor D_f \lfloor n\rfloor = D_e  (\lfloor m\rfloor D_f)^+ (\lfloor m\rfloor \lfloor n\rfloor)^+ \lfloor mn\rfloor.$$  Therefore, it is enough to show that $(\lfloor m\rfloor D_f)^+= (\lfloor m\rfloor D_f)^+(\lfloor m\rfloor \lfloor n\rfloor)^+$. We calculate:
\begin{align*}
(\lfloor m\rfloor D_f)^+&(\lfloor m\rfloor \lfloor n\rfloor)^+  = \,\, D_{f[m]^{-1}} D_{([m][n])^{+}}& (\text{by } \eqref{eq:aux_14a} \text{ with } e=1)\\
 & = \,\, D_{f[m]^{-1}([m][n])^{+}}  & (\text{by Lemma } \ref{lem:idempotents})\\
&  = \,\, D_{f[n]^+[m]^{-1}} & (\text{since }[m]^{-1}([m][n])^{+}  = [n]^+[m]^{-1})\\ 
& = \,\,  D_{f[m]^{-1}} =  (\lfloor m\rfloor D_f)^+, & (\text{since } f\leq [n]^{+} \text{ and  by } \eqref{eq:du})
\end{align*}
as needed.  

Let $m\in M$. We show that $\Psi\eta\lfloor m \rfloor = \lfloor m\rfloor$. This is equivalent to $\Psi([m]^{+}, m) = \lfloor m\rfloor$, which rewrites to $D_{[m]^+} \lfloor m\rfloor = \lfloor m\rfloor$. Because $D_{[m]^+} \lfloor m\rfloor \leq \lfloor m\rfloor$ and $(D_{[m]^+} \lfloor m\rfloor)^* = D_{[m]^+[m]} = D_{[m]} = \lfloor m\rfloor^*$, we are done.  We have shown that $\eta$ is an isomorphism. 

Recall from \eqref{eq:psir} that $\psi_{\varnothing} \colon {\mathcal{FR}}(M) \to {\mathcal{FI}}(M)$ is the $M$-canonical morphism of restriction monoids. In the sequel, we denote $\psi_{\varnothing}$ just by $\psi$.

\begin{corollary}\label{cor:consequences1a} 
The map $e\mapsto D_e$, $e\in E({\mathcal{FI}}(M))$, establishes an isomorphism between the semilattices $E({\mathcal{FI}}(M))$ and $P({\mathcal{FR}}(M))$. The inverse isomorphism is the map $\psi\vert_{P({\mathcal{FR}}(M))} \colon P({\mathcal{FR}}(M))\to E({\mathcal{FI}}(M))$. In particular,
$$\psi(D_e)=e \, \text{ for all } \, e\in E({\mathcal{FI}}(M)) \text{ and}
$$
$$D_{\psi(e)}=e \, \text{ for all } \, e\in P({\mathcal{FR}}(M)).  
$$
\end{corollary} 

\begin{proof} We first note that the semilattice $E({\mathcal{FI}}(M))$ is isomorphic to the semilattice $P(E({\mathcal{FI}}(M)) \rtimes M)$ via the map $e\mapsto (e,1)$ (see Subsection \ref{subs:structure_proper}). Because $\Psi$ is an isomorphism and $\Psi(e,1) = D_e$, it follows that the map $e\mapsto (e,1) \stackrel{\Psi}{\mapsto} D_e$ is an isomorphism from $E({\mathcal{FI}}(M))$ onto $P({\mathcal{FR}}(M))$. Furthermore, by Lemma \ref{lem:d_u} and Proposition~\ref{prop:crucial} we have $\psi(D_e) = e$ for all  $e\in E({\mathcal{FI}}(M))$. The statement follows.
\end{proof}

\subsection{Proof of Theorem \ref{th:main}: the case of an arbitrary $R$} We now proceed to the proof of Theorem~\ref{th:main} for the case where $R$ is an arbitrary set of admissible identities. 

\begin{proposition}\label{prop:crucial1a} Let $R$ be a set of admissible identities and $u,v \in {\mathcal{FI}}(M)$. If $u \mathrel{\gamma_R} v$ then $D_u \mathrel{\delta_R} D_v$.
\end{proposition}

\begin{proof} If $R=\varnothing$, the statement is trivial. We thus assume that $R\neq \varnothing$. It suffices to consider the case where $u$ and $v$ are $\gamma_R$-equivalent in one step. In view of Proposition \ref{prop:identities10}, we can assume that $u$ in $v$ are connected by a relation $w\cdot e = w\cdot f$, where $w\in X^*$ and $e,f \in  P(FR(X^*))$, or by its inverse. 

{\em Case 1.} Let first $u=p\,[w\cdot e]_{\iota_{{\mathcal{FI}}(M)},f}\,q$ and $v= p\,[w\cdot f]_{\iota_{{\mathcal{FI}}(M)},f}\,q$ where $f\colon X\to M$ is a map and $p,q\in {\mathcal{FI}}(M)$. Since $\iota_{{\mathcal{FI}}(M)}=\psi\iota_{{\mathcal{FR}}(M)}$ with $\psi$ being a morphism, Lemma~\ref{lem:useful26} implies that 
\begin{equation}\label{eq:aux7m}
[e]_{\iota_{{\mathcal{FI}}(M)},f} = \psi([e]_{\iota_{{\mathcal{FR}}(M)},f}).
\end{equation} We observe that
\begin{align*}
D_{p[w\cdot e]_{\iota_{{\mathcal{FI}}(M)},f}} & = D_{p[w]_{\iota_{{\mathcal{FI}}(M)},f}[e]_{\iota_{{\mathcal{FI}}(M)},f}} & (\text{by Definition } \ref{def:value})\\
& = D_{p[w]_{\iota_{{\mathcal{FI}}(M)},f}} D_{{[e]_{\iota_{{\mathcal{FI}}(M)},f}}} & (\text{by  Lemma } \ref{lem:idempotents})\\
& = D_{p[\hat{f}(w)]} D_{\psi([e]_{\iota_{{\mathcal{FR}}(M)},f})} & (\text{by Definition } \ref{def:value} \text{ and } \eqref{eq:aux7m})\\
& = (D_{p}\lfloor \hat{f}(w)\rfloor)^* [e]_{\iota_{{\mathcal{FR}}(M)},f} & (\text{by } \eqref{eq:du} \text{ and  Corollary } \ref{cor:consequences1a})\\
& = (D_{p}\lfloor \hat{f}(w)\rfloor [e]_{\iota_{{\mathcal{FR}}(M)},f})^* & (\text{by } \eqref{eq:frequent})\\
& = (D_{p}[w\cdot e]_{\iota_{{\mathcal{FR}}(M)},f})^*. & (\text{by Definition } \ref{def:value})
\end{align*}
Likewise, $D_{p[w\cdot f]_{\iota_{{\mathcal{FI}}(M)},f}} = (D_{p}[w\cdot f]_{\iota_{{\mathcal{FR}}(M)},f})^*$. Therefore, in view of Lemma \ref{lem:aux10}, we have $D_u \mathrel{\delta_R} D_v$, as needed. Similarly, one shows that $D_{u^{-1}} \mathrel{\delta_R} D_{v^{-1}}$.

{\em Case 2.} Let $u$ and $v$ be connected by the relation $(w\cdot e)^{-1} = (w\cdot f)^{-1}$. Then $u^{-1}$ and $v^{-1}$ are connected by the relation $w\cdot e = w\cdot f$, which reduces this case to the previous one.
\end{proof}
Therefore, the map ${\mathcal{FI}}(M) \to P({\mathcal{FR}}(M))$, $u\mapsto D_u$, gives rise to the induced map  ${\mathcal{FI}}_{R}(M) \to P({\mathcal{FR}}_{R}(M))$, also denoted by $u\mapsto D_u$.

The quotient map $\gamma_R^{\natural} \colon {\mathcal{FI}}(M) \to {\mathcal{FI}}_{R}(M)$, in turn, induces a quotient map of restriction monoids $E({\mathcal{FI}}(M))\rtimes M \to   E({\mathcal{FI}}_{R}(M)) \rtimes M$, given by $(e, m)\mapsto (\gamma_R^{\natural}(e), m)$.
It follows that there is an analogue of Proposition \ref{prop:generation} for any set $R$ of admissible identities: $E({\mathcal{FI}}_{R}(M))\rtimes M$ is $M$-generated via the map $\tau_R$. Consequently,  the morphism $\Psi$ of \eqref{eq:psi} induces the morphism
\begin{equation}\label{eq:psi10}
\Psi_R \colon E({\mathcal{FI}}_{R}(M))\rtimes M \to {\mathcal{FR}}_{R}(M), \,\, \Psi_R (e,m) = D_e \lfloor m\rfloor,
\end{equation}
that satisfies  $\Psi_R\eta_R = {\mathrm{id}}_{{\mathcal{FR}}_{R}(M)}$ (in particular, it is an isomorphism).  This completes the proof of Theorem~\ref{th:main}. 

The following consequence of Theorem~\ref{th:main} extends Corollary \ref{cor:consequences1a}.

\begin{corollary}\label{cor:consequences1c} Let $R$ be a set of admissible identities.
The map $e\mapsto D_e$, $e\in E({\mathcal{FI}}_R(M))$, establishes an isomorphism between the semilattices $E({\mathcal{FI}}_R(M))$ and $P({\mathcal{FR}}_R(M))$. The inverse isomorphism is given by the map $$\psi_R\vert_{P({\mathcal{FR}}_R(M))} \colon P({\mathcal{FR}}_R(M))\to E({\mathcal{FI}}_R(M)).$$ In particular,
$$\psi_R(D_e)=e \, \text{ for all } \, e\in E({\mathcal{FI}}_R(M)) \text{ and}
$$
$$D_{\psi_R(e)}=e \, \text{ for all } \, e\in P({\mathcal{FR}}_R(M)).  
$$
\end{corollary}

\begin{corollary} Let $R$ be a set of admissible identities and $M$ be a monoid. If the word problem in ${\mathcal{FI}}_{R}(M)$ is decidable, so is the word problem in ${\mathcal{FR}}_{R}(M)$.
\end{corollary}
\begin{proof}
Applying \eqref{eq:mov_proj} and \eqref{eq:e1}, we can rewrite any element of ${\mathcal{FR}}_{R}(M)$ in the form $e\lfloor m\rfloor$, where $e\in P({\mathcal{FR}}_{R}(M))$ and $m\in M$. Since $e=D_{\psi_R(e)}$, we have $e\lfloor m\rfloor=D_{\psi_R(e)}\lfloor m\rfloor$.  Theorem \ref{th:main} implies that $D_{\psi_R(e)}\lfloor m\rfloor = D_{\psi_R(f)} \lfloor m\rfloor$ if and only if $\psi_R(e)=\psi_R(f)$ in ${\mathcal{FI}}_{R}(M)$. If we know the deduction of $\psi_R(e)=\psi_R(f)$ in ${\mathcal{FI}}_{R}(M)$, proofs of Corollary~\ref{cor:equiv1}, Proposition~\ref{prop:crucial} and Proposition~\ref{prop:crucial1a} provide an algorithm  how to pass from $e=D_{\psi_R(e)}$ to $f=D_{\psi_R(f)}$ in ${\mathcal{FR}}_{R}(M)$.
\end{proof}

We conclude this section with an example demonstrating how the known structure of ${\mathcal{FI}}_{R}(M)$ sheds light onto the structure of ${\mathcal{FR}}_{R}(M)$.

\begin{example} {\em Let $M=\langle a\,  | \, a^2=a^3\rangle$ and $R=R_m$. The structure of ${\mathcal{FI}}_{R_m}(M)$ in this example is easily understood: it is the quotient of the free $[a]$-generated  inverse monoid by the relation $[a]^2=[a]^3$. In particular, its idempotents are described by one-rooted Munn trees (\cite{Munn}) over $[a]$ with at most two edges. It is thus easy to check that $$E({\mathcal{FI}}_{R_m}(M))=\{1, [a]^{*}, [a]^{+}, [a]^{*}[a]^{+}, ([a]^2)^{*}, ([a]^2)^{+}\}.$$

By Theorem \ref{th:main} $\eta_{R_m}\colon {\mathcal{FR}}_{R_m}(M) \to E({\mathcal{FI}}_{R_m}(M))\rtimes M$ is an isomorphism, thus one gets a handle on the structure of ${\mathcal{FR}}_{R_m}(M)$.
For example, because $$
\eta_{R_m}((\lfloor a \rfloor^{*} \lfloor a\rfloor)^{+}) = \eta_{R_m}((\lfloor a \rfloor^{*} \lfloor a\rfloor)^{*}) =  ([a]^*[a]^+,1),
$$
we have $(\lfloor a \rfloor^{*} \lfloor a\rfloor)^{+} = (\lfloor a\rfloor^{*}\lfloor a\rfloor)^{*}$.
On the other hand,
\begin{equation}\label{eq:distinct}
\lfloor a\rfloor^*\lfloor a\rfloor \neq (\lfloor a\rfloor^*\lfloor a\rfloor)^+,
\end{equation} because
$$
\eta_{R_m}(\lfloor a\rfloor^*\lfloor a\rfloor)  = ([a]^+,a)^*([a]^+,a) = ([a]^*[a]^+,a) \neq ([a]^*[a]^+,1) = \eta_{R_m}((\lfloor a \rfloor^{*} \lfloor a\rfloor)^{+}).
$$
In contrast to \eqref{eq:distinct}, in ${\mathcal{FI}}_{R_m}(M)$ we have:
\begin{equation}\label{eq:distinct1}
[a]^*[a] = ([a]^*[a])^+.
\end{equation}
Indeed, 
$$([a]^{*}[a])([a]^{*}[a])([a]^{*}[a])=[a]^{*}[a]^3 = [a]^{*}[a]^2 = [a]^{-1}[a]^3 =  [a]^{-1}[a]^2 = [a]^{*}[a],
$$
so that $([a]^{*}[a])^{-1}=[a]^{*}[a]$ by the uniqueness of the inverse element, which implies \eqref{eq:distinct1}. This means that the map $\psi_{R_m}$ in this example is not injective (cf. Proposition~\ref{prop:embedding}).
} 
\end{example}

\section{Special cases}\label{sect:special}

Let $M$ be a monoid and $R$ a set of admissible identities. 

\subsection{The case where $M$ embeds into a group} 
For a monoid $M$, we put $\lceil M \rceil=\{\lceil m \rceil \colon m\in M\}$. 
The free group $FG \lceil M \rceil$ is $M$-generated via the map 
$\iota_{FG \lceil M \rceil}\colon M \to FG \lceil M \rceil$,  $m\mapsto \lceil m \rceil$. 

We define ${\mathcal{FG}}(M)$ to be the quotient of $FG \lceil M \rceil$ by the congruence generated by the relations
\begin{equation}\label{eq:em19g}
\lceil  1\rceil  = 1;
\end{equation} 
\begin{equation}\label{eq:e1g}
\lceil  m\rceil  \lceil  n\rceil  =\lceil  mn\rceil  (\lceil  m\rceil  \lceil  n\rceil )^*, \,\,\, m,n\in M.
\end{equation}
Since $FG \lceil M \rceil$ is a group,  $(\lceil  m\rceil  \lceil  n\rceil )^*=1$, so that \eqref{eq:e1g} is equivalent to $\lceil  m\rceil  \lceil  n\rceil  =\lceil  mn\rceil $, $m,n\in M$. It follows that the $M$-canonical map $\iota_{{\mathcal{FG}}(M)}\colon M\to {\mathcal{FG}}(M)$ is a morphism. In particular, $\iota_{{\mathcal{FG}}(M)}$ satisfies any admissible identity. Thus, ${\mathcal{FG}}_R(M)={\mathcal{FG}}(M)$ for any set of admissible identities (where  ${\mathcal{FG}}_R(M)$ is defined along the lines of Definition  \ref{def:expansion_restr}).
The group ${\mathcal{FG}}(M)$ is known as the
 {\em free group over the monoid} $M$ \cite{CP} (also called the {\em fundamental group}  of $M$ \cite{MP}). 
 
\begin{lemma}\label{lem:quotient}  Let $R$ be a set of admissible identities. The maximum group quotient of  ${\mathcal{FI}}_R(M)$ is isomorphic to ${\mathcal{FG}}(M)$.
\end{lemma}

\begin{proof} The definitions imply that ${\mathcal{FG}}(M)$ is a quotient of ${\mathcal{FI}}_R(M)$. Let $\pi$ be the corresponding congruence on ${\mathcal{FI}}_R(M)$. 
We assume that $u,v\in {\mathcal{FI}}_R(M)$ are such that $u \mathrel{\pi} v$ and show that $u\mathrel{\sigma} v$.
Since $\pi$ is generated by the relations $[m][m]^{-1}=1$ and $[m][n] =[mn]$, $m,n\in M$, it is enough to consider the following cases.

{\em Case 1.} We assume that $u = p[m][m]^{-1}q$ and $v = pq$. Because $[m][m]^{-1}=[m]^{+}\leq 1$,  it follows that  $u\leq v$, so that  $u\mathrel{\sigma} v$.

{\em Case 2.} We assume that $u = p[m][n]q$ and $v = p[mn]q$. Because $[m][n]  \leq [mn]$,  it follows that  $u\leq v$, so that $u\mathrel{\sigma} v$. 

We have proved that $\pi\subseteq \sigma$, so that $\pi = \sigma$, by the minimality of $\sigma$.
\end{proof}

\begin{proposition}\label{prop:embedding} Let $M$ be a monoid and $R$ a set of admissible identities. The following statements are equivalent:
\begin{enumerate}
\item $M$ embeds into a group.
\item The morphism $\psi_R\colon {\mathcal{FR}}_{R}(M)\to {\mathcal{FI}}_R(M)$ of \eqref{eq:psir} is injective.
\end{enumerate}
\end{proposition}

\begin{proof} 
The definition of ${\mathcal{FG}}(M)$ implies that $M$ embeds into a group if and only if $M$ embeds into ${\mathcal{FG}}(M)$, that is,  the map $\iota_{{\mathcal{FG}}(M)}$ is injective.

(1) $\Rightarrow$ (2)  We assume that the map $\iota_{{\mathcal{FG}}(M)}$ is injective. By the proof of Theorem~\ref{th:main} (see \eqref{eq:psi10}), any element of ${\mathcal{FR}}_{R}(M)$ can be written in the form $D_e \lfloor m\rfloor$ where $e\leq [m]^{+}$. Let $\psi_R(D_e \lfloor m\rfloor) = \psi_R(D_f \lfloor n\rfloor)$ where $e\leq [m]^{+}$ and $f\leq [n]^{+}$. In view of Corollary \ref{cor:consequences1c} this yields
$e[m] = f[n]$. Since $e[m] \mathrel{\sigma} f[n]$, we obtain $\lceil m\rceil = \lceil n\rceil$, by Lemma~\ref{lem:quotient}. The assumption implies that $m = n$ whence $e[m]=f[m]$. But then $e=(e[m])^+=(f[m])^+=f$.  It follows  that $D_e \lfloor m\rfloor = D_f \lfloor n\rfloor$, so that the map $\psi_R$ is injective.

(2) $\Rightarrow$ (1) We assume that the map $\psi_R$ is injective. Let $m,n\in M$ be such that $\lceil m\rceil = \lceil n\rceil$. Then $[m] \mathrel{\sigma} [n]$ in ${\mathcal{FI}}_R(M)$. By the definition of $\sigma$ there are $e,f\in E({\mathcal{FI}}_R(M))$ such that $e[m] = f[n]$. Without loss of generality, we assume that $e\leq [m]^+$ and $f\leq [n]^+$. Observe that $e[m] = \psi_R\Psi_R(e,m)$ and, similarly, $f[n] = \psi_R\Psi_R(f,n)$. Since $\psi_R$ is injective and $\Psi_R$ is bijective, it follows that $(e,m)=(f,n)$ which yields $m=n$, as needed.
\end{proof}

We obtain the following result.

\begin{theorem} \label{th:consequences2} Let $M$ be a monoid that embeds into a group and $R$ a set of admissible identities. Then ${\mathcal{FR}}_{R}(M)$ is isomorphic to the restriction submonoid of ${\mathcal{FI}}_R(M)$ consisting of elements $e[m]$ where $e\in E({\mathcal{FI}}_R(M))$ and $m\in M$. This is precisely the $[M]$-generated restriction submonoid of ${\mathcal{FI}}_R(M)$. As a monoid, ${\mathcal{FR}}_R(M)$ is generated by  $E({\mathcal{FI}}_R(M))$ and $[M]$.
\end{theorem}

We remark that under the assumption of Theorem  \ref{th:consequences2}, ${\mathcal{FR}}_R(M)$ is a {\em full} restriction submonoid of ${\mathcal{FI}}_R(M)$, that is, it contains all idempotents of ${\mathcal{FI}}_R(M)$.

\begin{example} {\em Let $M=A^*$.  As already noted, ${\mathcal{FR}}_{R_m}(A^*)$ is isomorphic to $FR(A)$ and ${\mathcal{FI}}_{R_m}(A^*)$ is isomorphic to $FI(A)$. Theorem~ \ref{th:consequences2} implies that $FR(A)$ is isomorphic to the submonoid of $FI(A)$ generated by $E(FI(A))$ and $\{[a]\colon a\in A\}$. We thus recover the result of Fountain, Gomes and Gould~\cite{FGG} on the structure of $FR(A)$.}
\end{example}

Therefore, Theorem \ref{th:main} extends the result of~\cite{FGG} from $A^*$ to any monoid $M$ and from the set $R_m$  to any set $R$ of admissible identities.

\subsection{The case where $M$ is an inverse monoid} \label{sub:inv}
Let $S$ be an inverse monoid. In this section we provide models for ${\mathcal{FR}}_{R_m}(S)$ and ${\mathcal{FR}}_{R_s}(S)$.
Due to Theorem \ref{th:main}, it is enough to provide models for ${\mathcal{FI}}_{R_m}(S)$ and ${\mathcal{FI}}_{R_s}(S)$.
Since, obviously, ${\mathcal{FI}}_{R_m}(S)\simeq S$, we have the following statement.

\begin{proposition} Let $S$ be an inverse monoid. Then ${\mathcal{FR}}_{R_m}(S) \simeq  E(S)\rtimes S$. In particular, if $G$ is a group then ${\mathcal{FR}}_{R_m}(G)$ is a group isomorphic to $G$. 
\end{proposition}

\begin{remark}\label{rem:inv} {\em We note that $E(S)\rtimes S$  can be endowed with the structure of an inverse monoid, by putting $(e,s)^{-1} = ((es)^{*}, s^{-1})$. 
Let $\varoast$ and $\varoplus$ denote the unary operations of the induced restriction monoid structure. Observe that, unless $s^{-1}s=1$, $$(e,s)^{\varoast} = (e,s)^{-1}(e,s) = ((es)^{*},s^{-1}s) \neq ((es)^{*}, 1) = (e,s)^*$$
and, similarly, $(e,s)^{\varoplus}\neq (e,s)^+$.
It follows that, unless $S$ is a group, the image of $S$ under the morphism $\tau_{R_m}\colon S\to E(S)\rtimes S$, $s\mapsto (s^+,s)$, does not generate  $E(S)\rtimes S$ as an inverse monoid. For example, the element $(e,s)^* = ((es)^{*}, 1)$ does not belong to this image.}
\end{remark}

We now turn to the inverse monoid ${\mathcal{FI}}_{R_s}(S)$. The following result characterizes strong premorphisms from $S$ to inverse monoids.

\begin{proposition} \label{prop:prem8} Let $S,T$ be inverse monoids and $f\colon S\to T$, $s\mapsto f_s$, a premorphism. The following statements are equivalent:
\begin{enumerate}
\item $f$ is strong;
\item $f$ satisfies the following conditions: 
\begin{enumerate}
\item[(i)] for all $s\in S$: $f_{s^{-1}} = f_s^{-1}$;
\item[(ii)] for all $s,t\in S$: if $s\leq t$ then $f_s\leq f_t$.
\end{enumerate}
\end{enumerate}
\end{proposition}

\begin{proof}
(1) $\Rightarrow$ (2) We assume that $f$ strong. If $e\in E(S)$ then $f_ef_e = f_{e^2} f_e^*= f_e f_e^{-1} f_e = f_e$,
so that $f_e\in E(T)$. Let $s\in S$. Then:
$$
f_s^{*}f_{s^*} = f_s^{-1}f_s f_{s^{-1}s}= f_s^{-1} f_s^{+} f_s =  f_s^{-1}f_s = f_s^{*},
$$
so that $f_s^{*} \leq  f_{s^*}$.  Using this, we obtain:
$$
f_s f_{s^{-1}} f_s = f_s f_{s^*} f_s^{*} =  f_sf_s^{*} = f_s.
$$
It follows that $f_s$ and $f_{s^{-1}}$ are mutually inverse, so that $f_{s^{-1}} = f_s^{-1}$, by the uniqueness of the inverse element. We have shown condition (i).

To show condition (ii), we assume that $s\leq t$, that is, that $s=ts^{*}$. Then, in view of $f_s^{*} \leq  f_{s^*}$, we have:
$$
f_tf_{s^*} = f_{ts^*}f_{s^*}^{*} = f_sf_{s^*} = f_sf_s^*f_{s^*} = f_sf_s^* = f_s,
$$
so that $f_s\leq f_t$, as needed.  

(2) $\Rightarrow$ (1) We assume that $f$ satisfies conditions (i) and (ii) of part (2). Then, for any $s,t\in S$, we have:
\begin{equation*}
f_sf_t=f_sf_tf_t^{*} \leq f_{st}f_t^{*} = f_{st}f_{t^{-1}}f_t \leq 
f_{stt^{-1}}f_t \leq f_sf_t.
\end{equation*}
It follows that $f_sf_t=f_{st}f_t^{*}$, so that $f$ is right strong.
By symmetry, it is also left strong. Hence $f$ is strong.
\end{proof}

Proposition \ref{prop:prem8} combined with \cite[Theorem 6.10, Theorem 6.17]{LMS} leads to a presentation for $S^{pr}$, the Lawson-Margolis-Steinberg {\em generalized prefix expansion} of $S$ \cite{LMS}. 

\begin{theorem}\label{th:presentation} For any inverse monoid $S$,  $S^{pr}$ is isomorphic to ${\mathcal{FI}}_{R_s}(S)$. In particular, $S^{pr}$ is generated by the set $[S]=\{[s]\colon s\in S\}$, subject to the defining relations: 
{\em \begin{enumerate}
\item $[1]=1$;
\item  $[s][t] = [st][t]^{*}$, $[s][t] = [s]^{+}[st]$,  $s,t\in S$.
\end{enumerate}}
\end{theorem}

Another presentation for $S^{pr}$ was obtained by Buss and Exel in \cite{BE}.

\begin{remark}{\em It follows from the proof of Proposition~\ref{prop:prem8} that the statement of Proposition~\ref{prop:prem8}  (and thus of Theorem \ref{th:presentation}) remains valid in the semigroup setting, that is, in the absence of the identity element and the requirement that the identity element be preserved by a premorphism.}
\end{remark}

From Theorem \ref{th:main} and Theorem \ref{th:presentation} we obtain the following.

\begin{corollary}\label{cor:presentation_inv} Let $S$ be an inverse monoid. Then: 
$${\mathcal{FR}}_{R_s}(S) \simeq E(S^{pr})\rtimes S.$$
\end{corollary}

Using the known model for $S^{pr}$ \cite[Proposition 6.16]{LMS} this gives a model for ${\mathcal{FR}}_{R_s}(S)$.

There is an analogue of Remark \ref{rem:inv} for ${\mathcal{FR}}_{R_s}(S)$: it can be similarly endowed with the structure of an inverse monoid, however, unless $S$ is a group, $\tau_{R_s}(S)$ does not generate  $E({\mathcal{FI}}_{R_s}(S))\rtimes S$ as an inverse monoid.

\begin{corollary}\label{cor:group} Let $G$ be a group. Then ${\mathcal{FR}}_{R_s}(G)$ is an inverse monoid isomorphic to each of ${\mathcal{FI}}_{R_s}(G)$ and $\widetilde{G}^{{\mathcal R}}$.
\end{corollary}

\begin{proof} The isomorphism ${\mathcal{FR}}_{R_s}(G) \simeq {\mathcal{FI}}_{R_s}(G)$ follows from Theorem \ref{th:consequences2}. The isomorphism ${\mathcal{FI}}_{R_s}(G) \simeq \widetilde{G}^{{\mathcal R}}$ follows from Theorem \ref{th:presentation} and $G^{pr} \simeq \widetilde{G}^{{\mathcal R}}$.  
\end{proof}

\section*{Acknowlegement} The author is in debt to the referee for pointing out that admissible relations should be independent of a specific monoid $M$, which inspired me to introduce admissible identities. I am also grateful for the referee's comments and suggestions, which led to a considerable improvement of the presentation of my ideas and results.

\end{document}